\documentclass[runningheads]{llncs}
\usepackage[utf8]{inputenc}

\usepackage[T1]{fontenc}
\usepackage{ucs}
\usepackage{amsmath}
\usepackage{amsfonts}
\usepackage{amssymb}
\usepackage{enumitem}
\usepackage{xspace}
\usepackage{ stmaryrd }
\usepackage{graphicx}
\usepackage{geometry}
\usepackage{xcolor}
\usepackage{caption}
\usepackage{subcaption}
\usepackage{mathtools}
\usepackage{pdfpages}

\usepackage{listings}

\usepackage[appendix=inline]{apxproof}

\usepackage{xcolor}

\newcommand{\frto}[2]{
\begin{smallmatrix}
#2\\
#1
\end{smallmatrix}
}

\newcommand{\conca}{}
\usepackage{color}

\usepackage{soul}

\usepackage{datetime}

\newdateformat{monthyeardate}{%
	\monthname[\THEMONTH], \THEYEAR}

\DeclarePairedDelimiter\floor{\lfloor}{\rfloor}

\newcommand{\prfn}[2]{{\ensuremath {\Pi_{#1} \text{-} R ^{#2}}}\xspace}
\newcommand{\W}{\mathbb{W}}

\newcommand{\la}{\langle}
\newcommand{\ra}{\rangle}

\DeclareMathOperator{\ea}{EA}

\DeclareMathOperator{\con}{Con}

\DeclareMathOperator{\ewd}{EWD}

\DeclareMathOperator{\rfn}{RFN}
\DeclareMathOperator{\bgl}{GL}
\DeclareMathOperator{\glp}{GLP}

\DeclareMathOperator{\aca}{ACA}

\DeclareMathOperator{\is}{I\Sigma}

\DeclareMathOperator{\pa}{PA}
\DeclareMathOperator{\monid}{\top}
\newcommand{\w}{\mathbb{W}}

\newcommand{\utb}{\ensuremath{\textup{UTB}}\xspace}

\makeatletter
\DeclareFontFamily{OMX}{MnSymbolE}{}
\DeclareSymbolFont{MnLargeSymbols}{OMX}{MnSymbolE}{m}{n}
\SetSymbolFont{MnLargeSymbols}{bold}{OMX}{MnSymbolE}{b}{n}
\DeclareFontShape{OMX}{MnSymbolE}{m}{n}{
	<-6>  MnSymbolE5
	<6-7>  MnSymbolE6
	<7-8>  MnSymbolE7
	<8-9>  MnSymbolE8
	<9-10> MnSymbolE9
	<10-12> MnSymbolE10
	<12->   MnSymbolE12
}{}
\DeclareFontShape{OMX}{MnSymbolE}{b}{n}{
	<-6>  MnSymbolE-Bold5
	<6-7>  MnSymbolE-Bold6
	<7-8>  MnSymbolE-Bold7
	<8-9>  MnSymbolE-Bold8
	<9-10> MnSymbolE-Bold9
	<10-12> MnSymbolE-Bold10
	<12->   MnSymbolE-Bold12
}{}
\DeclareMathDelimiter{\ulcorner}
{\mathopen}{MnLargeSymbols}{'036}{MnLargeSymbols}{'036}
\DeclareMathDelimiter{\urcorner}
{\mathclose}{MnLargeSymbols}{'043}{MnLargeSymbols}{'043}
\makeatother

\makeatletter
\newcommand{\dotminus}{\mathbin{\text{\@dotminus}}}

\newcommand{\@dotminus}{%
	\ooalign{\hidewidth\raise1ex\hbox{.}\hidewidth\cr$\m@th-$\cr}%
}
\makeatother

\newtheorem{rem}{Remark}
\newtheorem{notation}{Notation}

\begin{document}
\title{Arithmetical and Hyperarithmetical Worm Battles\thanks{Partially supported by the FWO-FWF Lead Agency Grant G030620N and the Spanish Ministry of Science and Innovation, Grant PID2019-107667GB-I00.}}
%
%
\author{David Fern\'andez-Duque\inst{1}\orcidID{0000-0001-8604-4183} 
 \and
Joost J. Joosten\inst{2}\orcidID{0000-0001-9590-5045}
\and
Fedor Pakhomov\inst{3}\orcidID{0000-0002-9629-9259}
\and
Konstantinos Papafilippou\inst{4}\orcidID{0000-0002-2831-0575}
\and Andreas Weierman\inst{5}\orcidID{0000-0002-5561-5323}}
\authorrunning{Fern\'andez-Duque et al.}
%
\institute{Department of Mathematics WE16, Ghent University, Ghent, Belgium\\
\email{David.FernandezDuque@UGent.be} \and
Department of Philosophy, University of Barcelona, Catalonia, Spain.\\
\email{jjoosten@ub.edu} \and
Department of Mathematics WE16, Ghent University, Ghent, Belgium\\
\email{fedor.pakhomov@ugent.be} \and 
Department of Mathematics WE16, Ghent University, Ghent, Belgium\\
\email{Konstantinos.Papafilippou@UGent.be} \and
Department of Mathematics WE16, Ghent University, Ghent, Belgium\\
\email{Andreas.Weiermann@UGent.be}}
\maketitle              
\begin{abstract}

Japaridze's provability logic $\glp$ has one modality $[n]$ for each natural number and has been used by Beklemishev for a proof theoretic analysis of Peano aritmetic ($\pa$) and related theories.
Among other benefits, this analysis yields the so-called {\em Every Worm Dies} ($\ewd$) principle, a natural combinatorial statement independent of $\pa$.
Recently, Beklemishev and Pakhomov have studied notions of provability corresponding to transfinite modalities in $\glp$.
We show that indeed the natural transfinite extension of $\glp$ is sound for this interpretation, and yields independent combinatorial principles for the second order theory $\aca$ of arithmetical comprehension with full induction.
We also provide restricted versions of $\ewd$ related to the fragments $\is_n$ of Peano arithmetic.
In order to prove the latter, we show that standard Hardy functions majorize their variants based on tree ordinals.

\end{abstract}
\section{Introduction}

It is an empirically observed phenomenon that `natural' theories are linearly ordered by strength, suggesting that this strength could be quantified in some fashion.
Much as real numbers are used to measure e.g.~the distance between two points on the plane, proof theorists use ordinal numbers to measure the power of formal theories~\cite{RathjenArt}.
The precise relationship between these theories and their respective ordinals may be defined in various ways, each with advantages and disadvantages.
One relatively recent and particularly compelling way to assign ordinals to a theory $T$ lies in studying hierarchies of iterated consistency or reflection principles for a weaker base theory $B$ that are provable in $T$.
The work of Beklemishev~\cite{Beklemishev:2004:ProvabilityAlgebrasAndOrdinals} has shown how provability logic, particularly Japaridze's polymodal variant $\glp$ \cite{Japaridze:1988}, provides an elegant framework for analyzing theories in this fashion.
$\glp$ is a propositional logic which has one modality $[n]$ for each natural number.
The expression $[n]\varphi$ is read {\em $\varphi$ is $n$-provable,} where $n$-provability is defined by allowing any true $\Pi_n$ sentence as an axiom.
Dually, $\langle n\rangle\varphi$ denotes the {\em $n$-consistency} of $\varphi$, which is equivalent to the schema stating that all $\Sigma_n$ consequences of $\varphi$ are true, also known as {\em $\Sigma_n$-reflection.}

This approach to ordinal analysis is based on special elements of the logic --the so-called \emph{worms}.
Formally, worms are expressions of the form $\langle n_1\rangle\ldots \langle n_m\rangle \top$, representing iterated reflection principles.
However, worms can be interpreted in many ways: formulas of a logic, words over an infinite alphabet, special fragments of arithmetic \cite{Beklemishev:1997:InductionRules,Leivant:1983:OptimalityOfInduction}, Turing progressions \cite{Joosten:2016:TuringTaylorExpansion,Turing:1939:TuringProgressions}, worlds in a special model for the closed fragment of $ \glp $, ordinals \cite{FernandezJoosten:2014:WellOrders}, and also iterations of special functions on ordinals \cite{FernandezJoosten:2012:Hyperations}.
These interpretations of worms allowed Beklemishev~\cite{Beklemishev:2004:ProvabilityAlgebrasAndOrdinals} to give an ordinal analysis of Peano arithmetc ($\pa$) and related systems, yielding as side-products a  classification of provably total recursive functions, consistency proofs, and a combinatorial principle independent of $\pa$, colloquially called {\em Every Worm Dies.}

Recently, Beklemishev and Pakhomov~ \cite{BeklemishevPakhomov:2019:GLPforTheoriesOfTruth} extended the method of ordinal analysis via provability logics to predicative systems of second order arithmetic.
It is important to investigate if said analysis also comes with the expected regular side-products for theories beyond the strength of $\pa$.
This paper is a first exploration in this direction, expanding on their analysis to provide combinatorial principles independent of standard extensions of $\pa$.

Beklemishev and Pakhomov's analysis involves notions of provability naturally corresponding to modalities $[\lambda]$ for $\lambda\geq \omega$ in the natural transfinite extension of $\glp$.
This extension is denoted $\glp_\Lambda$ \cite{BeklemishevFernandezJoosten:2014:LinearlyOrderedGLP}, where $\Lambda$ is the supremum of all modalities allowed.
The model theory of the logics $\glp_\Lambda$ has been extensively studied \cite{Fernandez:2012:TopologicalCompleteness,FernandezJoosten:2013:ModelsOfGLP}, as have various proof-theoretic interpretations and applications \cite{Joosten:2013:AnalysisBeyondFO,CordonFernandezJoostenLara:2017:PredicativityThroughTransfiniteReflection,FernandezJoosten:2018:OmegaRuleInterpretationGLP,Joosten:2020:MunchhausenProvability}.
Our first result is that $\glp_\Lambda$ is also sound for the notions of provability employed in \cite{BeklemishevPakhomov:2019:GLPforTheoriesOfTruth}.

The soundness of $\glp_\Lambda$ allows us to develop combinatorial principles in the style of Beklemishev \cite{Beklemishev:2004:ProvabilityAlgebrasAndOrdinals}, an effort that was initiated in Papafilippou's master thesis \cite{Papafillipou:2020:MastersThesis}.
We consider variants of the {\em Every Worm Dies} principle, denoted $\ewd^\Lambda$ for suitable $\Lambda$.
Our further main results are that over elementary arithmetic ($\ea$),  $\ewd^{\omega^2}$ is equivalent to the $1$-consistency of $\aca$, and that the principles $\ewd^{n+1}$ are equivalent to the $1$-consistency of $\is_n$.

\section{Preliminaries}

For first order arithmetic, we shall work with theories with identity in the language $\mathcal{L}_{\pa} :=\{ 0, S, +, \cdot, {\rm exp}\}$, with ${\rm exp}$ being the unary function for $x \mapsto 2^x$. We define $\Delta_0 = \Sigma_0 = \Pi_0$-formulas as those whose quantifiers occur in the form $\forall x{<}t \ \varphi$ or $\exists x{<}t \ \varphi$. Then, $\Sigma_{n+1}/\Pi_{n+1}$-formulas are inductively defined to be those of the form $\exists  x  \varphi/\forall  x   \varphi$, where $\varphi$ is a $\Pi_n/\Sigma_n$-formula, respectively. 
We may extend the above classes with a new predicate $P$ by treating it as an atomic formula: the resulting classes are denoted $\Pi_n (P) $, $\Sigma_n (P) $, etc.
More generally, for an extension $\mathcal{L} \supsetneq \mathcal{L}_{\pa}$ of the language of $\pa$ with new predicate symbols, we write $\Pi_n ^\mathcal{L}$, $\Sigma_n ^\mathcal{L}$, etc.~to denote the corresponding classes of formulas with any new predicate symbols of $\mathcal{L}$ treated as atoms.

\emph{Elementary Arithmetic} ($\ea$) or \emph{Kalmar Arithmetic} contains the basic axioms describing the non-logical symbols together with the induction axiom $ \text{I}_\varphi $ for every $\Delta_0$-formula $ \varphi $, which as usual denotes $  \text{I}_\varphi {:=} \varphi(0) \, \wedge \, \forall \, x \ \big( \varphi(x) \rightarrow \varphi(S(x)) \big) \rightarrow \forall \, x \ \varphi(x)$. For a given class of formulas $\Gamma$, we denote by $I\Gamma$ the theory extended $\ea$ with induction for all $\Gamma$-formulas. By $\ea^+$ we denote the extension of $\ea$ by the axiom expressing the totality of the  \emph{super-exponentiation} function $2^x _y$, which is defined inductively as: $2^x _0{ :=} x; \ \ 2^x _{n+1}{:=} 2^{2^x _n}$. Finally $\pa$ can be seen as the union of all $\is_n$ for every $n$. 
A theory $S$ of a language $\mathcal{L} \supseteq \mathcal{L}_{\pa}$ is elementary axiomatizable if there is a $ \Delta_0 $-formula $ Ax_S(x) $ that is true iff $ x $ is the code of an axiom of $ S $. By Craig's trick, all c.e. theories have an equivalent that is elementary axiomatizable.

The language of second order arithmetic is the extension of the language of first order arithmetic $ \mathcal{L}_{\pa} $ by the addition of second order variables and parameters and the predicate symbol $ \in $. The expression $ t {\in} X $ is an atomic formula where $ t $ is a term and $ X $ a second order variable. We add no symbol for the second order identity, but it can be defined via extensionality.
\begin{definition}[$ \aca $]
	The theory $ \aca $ is a theory in the language of second order arithmetic that extends $ \pa $ by the induction schema for all second order formulas and the comprehension schema:
	$ \exists \, Y \, \forall \, x \ (x {\in} Y \leftrightarrow \varphi(x) ),
	$
	for every arithmetical formula with possibly both first and second order parameters (excluding $ Y $).
\end{definition}

We may represent ordinals within arithmetic as pairs $\langle \Lambda,{<_\Lambda}\rangle$, where $\Lambda\subseteq \mathbb N$ and $<_\Lambda\subseteq \Lambda\times\Lambda$ are defined by $\Delta_0$ formulas (or, to be precise, $\Delta_0({\rm exp})$ formulas, as we allow exponentiation in our language).
Ordinals represented in this way are {\em elementarily presented;} any computable ordinal may be elementarily presented by Craig's trick (see~\cite{Beklemishev:2004:ProvabilityAlgebrasAndOrdinals}).
We notationally identify $\Lambda$ with $\langle \Lambda,{<_\Lambda}\rangle$ and may write $\lambda<\Lambda$ instead of $\lambda\in \Lambda$.
It is convenient to assume that limit ordinals below $\Lambda$ are equipped with {\em fundamental sequences,} i.e.~increasing sequences $\langle \lambda[n]\rangle_{n\in\mathbb N}$ which converge to $\lambda$; we also assume that fundamental sequences are elementary (i.e., have a $\Delta_0$ graph).
We also set $0[n] = 0$ and $(\alpha+1)[n] = \alpha$.
In second order logic we may express the property ``$\Lambda$ is well ordered,'' and it is well known that for large $\Lambda$, principles of this form are not provable over weak theories (see e.g.~\cite{Pohlers:2009:PTBook}).

\begin{definition}
	For $\Lambda$ an ordinal, the logic $\glp_\Lambda$ is the propositional modal logic with a modality $[\alpha]$ for each $\alpha<\Lambda$. Each $[\alpha]$ modality satisfies the $\bgl$ identities given by all tautologies, distribution axioms $[\alpha] (\varphi \to \psi) \to ([\alpha]\varphi \to [\alpha]\psi)$, L\"ob's axiom scheme $[\alpha] ([\alpha] \varphi \to \varphi ) \to [\alpha]\varphi$ and the rules modus ponens and necesitation $\varphi/[\alpha]\varphi$. The interaction between modalities is governed by two schemes, monotonicity $[\beta]\varphi \to [\alpha] \varphi$ and, negative introspection $\la \beta \ra \varphi \to [\alpha]\la \beta\ra \varphi$ where in both schemes it is required that $\beta < \alpha < \Lambda$. As usual $\la \alpha \ra \varphi$ is a shorthand for $\neg [ \alpha ] \neg \varphi$.
\end{definition}

Note that $\glp_\Lambda$ is well defined for any linear order $\Lambda$~\cite{BeklemishevFernandezJoosten:2014:LinearlyOrderedGLP}, but in this paper we do not consider ill-founded $\Lambda$.
The closed fragment of $\glp_\Lambda$ suffices for ordinal analyses and worms are its backbone.

\begin{definition}\label{definition:WormStuff}
	The class of \emph{worms} of $\glp_\Lambda$ is denoted $\W^\Lambda$ and defined by $\top{\in} \W^\Lambda$, and $A{\in} \W^\Lambda \wedge \alpha{<}\Lambda \Rightarrow \la \alpha\ra A {\in} \W^\Lambda$. By $\W^\Lambda_\alpha$ we denote the set of worms where all occurring modalities are at least $\alpha$. We define an order $<_\alpha$ for each $\alpha{<}\Lambda$ by setting $A{<_\alpha}B $ if $\glp_\Lambda {\vdash} B {\to} \la \alpha \ra A$. 
\end{definition}

It will be convenient to introduce notation to compose and decompose worms.
Let us write $\alpha$ instead of $\langle\alpha\rangle$ when this does not lead to confusion.
For worms $A$ and $B$ we define the concatenation $A{\conca} B$ via $\top {\conca} B{:=}B$ and $( \alpha   A){\conca} B  {:=}  \alpha   (A{\conca} B)$.
We define the $\alpha$-head $h_\alpha$ of $A$ inductively: $h_\alpha(\top){:=}\top$; $h_\alpha (  \beta A){:=}\top$ if $\beta{<}\alpha$, and $h_\alpha ( \beta A){:=} \beta h_\alpha(A)$ otherwise. Likewise, we define the $\alpha$-remainder $r_\alpha$ of $A$ as $r_\alpha(\top){:=}\top$ and, $r_\alpha ( \beta A){:=} \beta A$ if $\beta{<}\alpha$ and $r_\alpha ( \beta A){:=}r_\alpha (A)$ otherwise. We define the head $h$ and remainder $r$ of $ \alpha A$ as $h( \alpha A){:=}h_\alpha ( \alpha A)$ and $r( \alpha A){:=}r_\alpha ( \alpha A)$. Further, $h(\top){:=}r(\top){:=}\top$.

\begin{lemma} \label{lem: basic properties of glp concerned with building up}
	The following formulas are derivable in $ \glp_\Lambda $:
	\begin{itemize}
		\item[(i)] If $ \alpha \leq \beta $ and $ A \in \w^\Lambda $, then $ \glp_\Lambda\vdash \beta \alpha A \to \alpha A $;
		\item[(ii)] If $ \alpha < \beta $, then $ \glp_\Lambda\vdash \beta \varphi \wedge \alpha \psi \leftrightarrow \beta (\varphi \wedge \alpha \psi) $;
		\item[(iii)] If $ A \in \w_{\alpha+1} ^\Lambda $, then $ \glp_\Lambda\vdash AC \wedge \alpha B \leftrightarrow A(C \wedge \alpha B) $;
		\item[(iv)] If $ A \in \w_{\alpha+1} ^\Lambda $, then $ \glp_\Lambda\vdash A \wedge \alpha B \leftrightarrow A \alpha B $.
	\end{itemize}
\end{lemma}
The proof of which follows successively from the axioms of $ \glp_\Lambda $, details for which can be found in \cite{Beklemishev:2005:Survey} and \cite{BeklemishevFernandezJoosten:2014:LinearlyOrderedGLP}. With this lemma in our toolbelt, we can prove the following proposition which will be of use to us later as we present worm battles.
Below, the length of a worm $\left| B \right|$ is defined inductively as: $\left| \monid \right| = 0$ and $\left| \alpha B \right|= 1 + \left| B \right|$.
\begin{proposition} \label{prop: <a>^n+1 implies smaller A with <a>^n }
	The following is provable over $\ea$. Let $ n\in\mathbb N $, $ \alpha <\Lambda $ be ordinals, $ A \in \w^{\alpha} $, and $ B \in \w ^{\alpha + 1} $ be such that $ \left| B \right| \leq n $.
	Then,  
	$ \glp_\Lambda \vdash \la \alpha \ra^{n+1}\top \rightarrow A B. 
	$
\end{proposition}
\begin{toappendix}

\begin{proof}
	We will prove this fact through two external inductions, first we will show that for every $ n $ and $ B $ satisfying the above conditions, $ \glp_\Lambda \vdash \la \alpha \ra^{n}\top \rightarrow B $. If $ n = 0 $, then it is clear. Assume now that it holds for $ n=k $. Let $ B \in \w ^{\alpha + 1} $ with $ \left| B\right| \leq k $ and let $ \beta \leq \alpha $, then in $\glp_\Lambda$, $ \la \alpha \ra^{k+1}\top \vdash \la \alpha \ra^{\left|B\right|+1}\top \vdash \la \alpha \ra\la \alpha \ra^{\left|B\right|}\top \vdash \la \alpha \ra B\vdash \la \beta \ra B$, where the first step uses at most $k $ applications of the \textbf{4} axiom.
	
	Now we will perform an external induction on $ \left|A\right| $. If $ A{=} \beta $ for some $ \beta < \alpha $, then we fall in the case of the previous induction.	If $ A {=} \la\beta\ra C  $, where $ \beta < \alpha $ and $ \glp_\Lambda \vdash \la \alpha \ra^{n+1}\top \rightarrow C B $, then in $\glp_\Lambda$, $ \la \alpha \ra^{n+1}\top \vdash (C B \wedge \la \alpha \ra \top)\vdash \la \alpha \ra CB$, using Lemma \ref{lem: basic properties of glp concerned with building up}.
\end{proof}

\end{toappendix}

From \cite{Beklemishev:2004:ProvabilityAlgebrasAndOrdinals} we know that $\la \W^\omega_n{/}{\equiv}, <_n\ra\, {\cong}\, \la \varepsilon_0, < \ra$, so that worms (modulo provable equivalence) can be used to denote ordinals. One can find analogs of fundamental sequences for ordinals by defining $Q^\alpha _0(\varphi){:=}\la\alpha\ra \varphi$; $Q^\alpha_{k+1}(\varphi){:=}\la \alpha\ra \big( \varphi \wedge Q^\alpha _k(\varphi)\big)$. By an easy induction on $k$ one sees that $Q^\alpha_{k+1}(A) <_\beta  \la \alpha{+}1\ra A$ for any $\beta\leq\alpha$ yielding a so-called step-down function.

This step-down function can be rewritten to get a more combinatorial flavour reminiscent of the Hydra battle. To this end, we define the \emph{chop-operator} $c$ on worms that do not start with a limit ordinal by $c(\top){:=}\top; c(\la 0\ra A){:=}A$ and, $c(\la \alpha+1\ra A){:=}\la \alpha\ra A$. Now we define a stepping down function based on a combination of chopping a worm, the worm growing back and using a given and fixed fundamental sequence of the limit ordinals occurring in $\glp_\Lambda$ for countable $\Lambda$. 

\begin{definition}[Step-down function]
For any number $k$ let $A\llbracket k\rrbracket{:=}c(A)$ for $A{=}\top$ or $A{=}0B$, $A\llbracket k\rrbracket{:=}\big( c(h(A))\big)^{k+1}\conca r(A)$ for $A = \la \alpha + 1 \ra B$ and $A\llbracket k\rrbracket{:=} \la \lambda [k] \ra B$ for $A = \la \lambda \ra B$ where $\lambda$ is a limit ordinal and $\lambda [k]$ is the $k$-th element of the fundamental sequence of $\lambda$.
\end{definition}
The definition above relates to the functions $Q^\alpha _k $ and serves as a way to produce fundamental sequences of worms inside $\la \W^\Lambda /\equiv, <_0 \ra$. With an easy induction on $k$, one can prove the following:
\begin{lemma}
	If $ A,B \in \w^\Lambda $ and $ A = \la \alpha {+} 1 \ra B $ then for every $ k \in \mathbb{N} $ we have that, 
	$$ \glp_\Lambda \vdash Q^\alpha _k (B) \leftrightarrow (\alpha h_{\alpha+1}(B))^{k+1} r_{\alpha+1}(B) .
	$$
\end{lemma}
Then, assuming that there is an elementary coding of the ordinals present according to $\ea$, the following is provable over $\ea$:
\begin{corollary}\label{cor: relationship between double bracket and double angle in GLP}
    For any $ k \in \mathbb{N}  $ and $ A \in \w^\Lambda $ with $ A \neq \top $, we have $ A\llbracket k \rrbracket <_0 A $.
\end{corollary}

\begin{toappendix}

\begin{proof}
    Since for every natural number $k$:
    $ \glp_\Lambda \vdash \la \alpha {+}1 \ra B \rightarrow Q^\alpha _k (B),
	$
	for $A = \la \alpha {+} 1 \ra B$, in $\glp_\Lambda $, $ A \vdash (\alpha h_{\alpha+1}(B))^{k+2} r_{\alpha+1}(B)
	\vdash\alpha A\llbracket k \rrbracket
	\vdash  0 A\llbracket k \rrbracket$.
	The limit stage follows by the monotonicity axiom of $\glp_\Lambda$.
\end{proof}

\end{toappendix}

Given a worm $A\in W^\Lambda$, we now define a decreasing sequence (strictly as long as we have not reached $\top$) by $A_0:=A$ and $A_{k+1}{:=}A_k\llbracket k+1\rrbracket$. We now define the principle $\ewd^\Lambda$ standing for \emph{Every Worm Dies} as an arithmetisation of $\forall A \in W^\Lambda \exists k A_k {=}\top$. Note that here the worms are coded as sequences of ordinals which we achieve by assuming that $\Lambda$ is elementarily presented.\\

The modalities of $\glp_\omega$ can be linked to arithmetic by interpreting $\la n\ra \varphi$ for a given c.e. theory $S\supseteq \ea $ as the finitely axiomatisable scheme $\Sigma_n\text{-}\rfn(S{+} \varphi^*): = \{ \Box_{S + \varphi^*} \sigma \to \sigma \mid \sigma \in \Sigma_n\} \ \equiv \Pi_{n+1}\text{-}\rfn(S{+} \varphi^*)$. The $\Box_{S}$ denotes the standard arithmetisation of formalised provability for the theory $S$ and $\varphi^*$ denotes an interpretation of $\varphi$ in arithmetic, mapping propositional variables to sentences, commuting with the connectives and, translating the $\la n \ra$ as above. This interpretation is used to classify the aforementioned first order theories of arithmetic.

\begin{theorem}[Leivant, Beklemishev \cite{Leivant:1983:OptimalityOfInduction,Beklemishev:1997:InductionRules}]
	Provably in $\ea^+$, for $n{\geq}1$: 
	\[
	\is_{n} \equiv \Sigma_{n+1}\text{-}\rfn(\ea) .\]
\end{theorem}
Further known results involving partial reflection are given by assessing the totality of certain functions.
For a $\Sigma_1$-definable function $f$, by $f{\downarrow}$ we denote the arithmetical sentence $\forall x \exists y f(x)=y$ stating that $f$ is defined everywhere and likewise, by $f(x){\downarrow}$ we denote $\exists y f(x)=y$.
\begin{lemma} (\cite{Beklemishev:2005:Survey})\label{lem: f^x (x) halts iff 1-con that f halts}
	Let $ f $ be a $\Sigma_1$-definable function that is non-decreasing and $ f(x)\geq 2^x $. Then,
	$$ \ea \vdash \lambda x. {f^{(x)}(x)} {\downarrow} \, \leftrightarrow \left\langle 1\right\rangle _{\ea} f{\downarrow} .$$
\end{lemma}
	
If we substitute $ f $ with $ {\rm exp}$, we get:
\begin{corollary} \label{cor: ea^+ is equiv to 1-con(ea)}
	Provably in $\ea$, we have $ \ea^+ \equiv \ea + \, \Pi_2 \textnormal{-}\rfn(\ea) $.
\end{corollary}

\section{Arithmetical Soundness of $\glp_\Lambda$}

In our interest of expanding the worm principle, we have to first expand the interpretation of $\glp$ in arithmetic for modalities $[\alpha]$ where $\alpha \geq \omega$.

Let $\mathcal{L}$ be a language of arithmetic with or without a unary predicate ${\sf T}$ and let $S$ be a c.e. theory extending $\ea$ in a language extending $\mathcal{L}$.
We will prove arithmetical soundness of $\glp_\Lambda$ for a particular interpretation for which most of the work has already been done in \cite{BeklemishevPakhomov:2019:GLPforTheoriesOfTruth} by proving arithmetical soundness for the weaker system of ${\sf RC}_\Lambda$. As such, we will call onto many results from that paper, starting with some properties of partial reflection in potentially extended languages of arithmetic.
\begin{lemma}
    For all sentences $\varphi, \psi \in \mathcal{L}$ and for every $n\geq 0$, the following hold provably in $\ea$:
    \begin{itemize}
        \item If $S \vdash \varphi  \to \psi$ then $\Pi_{n+1} ^\mathcal{L}$-$\rfn (S+ \varphi) \vdash \Pi_{n+1} ^\mathcal{L}$-$\rfn (S+\psi)$;
        \item $\Pi_{n+1} ^\mathcal{L}$-$\rfn (S+\varphi) \vdash \varphi $ if $\varphi \in \Pi_{n+1}^\mathcal{L}$;
        \item $\Pi_{n+1} ^\mathcal{L}$-$\rfn (S+\varphi) \vdash \Diamond_S \varphi$.
    \end{itemize}
\end{lemma}
It is known that $\Pi_{n+1} ^\mathcal{L}$-$\rfn(S)$ is finitely axiomatizable over $\ea$ for $\mathcal{L} = \mathcal{L}_{\pa}$, which is achieved by using \emph{truth-definitions} for $\Pi_{n+1} ^{\mathcal{L}_{\pa}} $-formulas. For $\mathcal{L} \supsetneq \mathcal{L}_{\pa} $, we have the following properties for truth definitions (Theorems 12 \& 13, \cite{BeklemishevPakhomov:2019:GLPforTheoriesOfTruth}):

\begin{theorem}
    Let $\mathcal{L}$ be finite. There is a $\Pi_1 ^\mathcal{L}$-formula ${\sf Tr}$ such that for all $\Delta_0 ^{\mathcal{L}}$-formulas $\varphi(\vec{x})$,
    \begin{itemize}
        \item $\ea \vdash \forall \, \vec{x} \ \big( {\sf Tr} (\varphi (\vec{x})) \to \varphi (\vec{x}) \big)$;
        \item $\ea^\mathcal{L} \vdash \forall \, \vec{x} \ \big( {\sf Tr} (\varphi (\vec{x})) \leftrightarrow \varphi (\vec{x}) \big)$.
    \end{itemize}
    Let $\Gamma$ be either $\Pi_n ^\mathcal{L}$ or $\Sigma_n ^\mathcal{L}$ for $n>0$, then there exists a $\Gamma$-formula $ {\sf Tr}_\Gamma $ such that for each $\Gamma$-formula $\varphi(\vec{x})$,
    \[\ea^\mathcal{L} \vdash \forall \, \vec{x} \ \big( {\sf Tr}_\Gamma (\varphi (\vec{x})) \leftrightarrow \varphi (\vec{x}) \big).
    \]
\end{theorem}

For languages extending the language of arithmetic, we require a way to finitely axiomatize $\Delta_0 ^\mathcal{L}$-induction, which is given for finite $\mathcal{L}$ (Lemma 4.2 in \cite{Enayat2019-ENATDA}). Over $\ea$ we have the following theorem.
\begin{theorem}[Thm 3, \cite{BeklemishevPakhomov:2019:GLPforTheoriesOfTruth}]\label{thm:finite axiomatizability of rfn}
    For finite $\mathcal{L}$ the schema $\Pi_{n+1} ^\mathcal{L}$-$\rfn (S)$ is finitely axiomatizable by
    \[ i\delta^\mathcal{L} \wedge \forall \varphi \in \Pi_{n+1}^\mathcal{L} \, \big(\Box_S \varphi  \to {\sf Tr}_{\Pi_{n+1} ^\mathcal{L}} (\varphi)\big),
    \]
    where $i\delta^\mathcal{L}$ is a $\Pi_1 ^\mathcal{L}$-axiomatization of $I\Delta_0 ^\mathcal{L}$ and ${\sf Tr}_{\Pi_{m} ^\mathcal{L}}$ is the truth definition for $\Pi_m ^\mathcal{L}$-formulas.
\end{theorem}
From here on, we will be using $\Pi_{n+1} ^\mathcal{L}$-$\rfn (S)$ and the formula axiomatizing it interchangeably where applicable. By $i\delta^\mathcal{L}$ we will always denote the $\Pi_1 ^\mathcal{L}$-axiomatization of $I\Delta_0 ^\mathcal{L}$. 

\begin{notation}
    Let $\mathcal{L}$ be finite, then given a formula $\varphi \in \mathcal{L}$, we write $[n]_S ^\mathcal{L} \varphi$ as shorthand for \\ $\exists \, \theta {\in} \Sigma_{n+1} ^\mathcal{L}  \big({\sf Tr}_{\Sigma_{n+1} ^\mathcal{L}} (\theta) \wedge \Box_S (\theta \to \varphi)\big) $.
\end{notation}

The lemma below corresponds to the distributivity axiom L1. and we will be using it to prove an arithmetical soundness of $\glp_\Lambda$.

\begin{lemma}\label{lem: D_0 induction on finite language gives distributivity}
   If $\mathcal{L}$ is finite, then 
   $ \ea^\mathcal{L} \vdash [n]_S ^\mathcal{L} (\varphi\to \psi) \to ( [n]_S ^\mathcal{L} \varphi \to [n]_S ^\mathcal{L} \psi) $.
\end{lemma}
\begin{proof}
Working within $\ea ^\mathcal{L}$, assume that $\exists \, \theta_1 \in \Sigma_{n+1} ^\mathcal{L} \big({\sf Tr}_{\Sigma_{n+1} ^\mathcal{L}} (\theta_1) \wedge \Box_S (\theta_1 \to (\varphi\to \psi))\big) $ and ${\exists \, \theta_2 {\in} \Sigma_{n+1} ^\mathcal{L}}\ \big({\sf Tr}_{\Sigma_{n+1} ^\mathcal{L}} (\theta_2) \wedge \Box_S (\theta_2 \to \varphi)\big) $. Since $\ea^\mathcal{L}\vdash {\sf Tr}_{\Sigma_{n+1} ^\mathcal{L}}(\varphi) \leftrightarrow \varphi $ for every $\Sigma_{n+1} ^\mathcal{L}$-formula $\varphi$, it is then given that $\ea^\mathcal{L} \vdash {\sf Tr}_{\Sigma_{n+1} ^\mathcal{L}} (\theta_1\wedge \theta_2) \leftrightarrow {\sf Tr}_{\Sigma_{n+1} ^\mathcal{L}} (\theta_1) \wedge {\sf Tr}_{\Sigma_{n+1} ^\mathcal{L}} (\theta_2) $. Thus we get ${\sf Tr}_{\Sigma_{n+1} ^\mathcal{L}} (\theta_1 \wedge\theta_2)  \wedge \, \Box_S \big((\theta_1\wedge\theta_2) \to \psi\big)$.
\end{proof}
We will be focusing on languages extending that of arithmetic via the addition of so-called truth predicates. These are unary predicates with the purpose of expressing the truth of formulas $-$a task achieved by expanding our base theories of arithmetic with the theory of the \emph{Uniform Tarski Biconditionals}.
\begin{definition}
    Let $\utb_\mathcal{L}$ be the $\mathcal{L} \cup {\sf T}$ theory --where ${\sf T}$ is a unary truth predicate not in $\mathcal{L}$-- axiomatized by the schema
    $\forall \vec{x} \big(\varphi(\vec{x}) \leftrightarrow {\sf T}( \ulcorner\varphi(\dot{\vec{x}}) \urcorner)\big)$,  for every $\mathcal{L}$-formula $ \varphi$.
\end{definition}
This process of extending the base language via the addition of truth predicates can be iterated over ordinals.
For that we assume that given an ordinal $\Lambda$ there are $\Delta_0$-formulas in the base language of $\rm EA$; $x < \Lambda$ and $x \leq _\Lambda y$, roughly expressing that ``$x$ is the code of an ordinal in $\Lambda"$ and ``$x, y$ code ordinals $\alpha, \beta$ with $\alpha \leq \beta"$ respectively. More formally, we want the following to hold:
\begin{itemize}
    \item For every ordinal $\alpha < \Lambda $, it holds that $\mathbb{N}\vDash \ulcorner\alpha\urcorner <_\Lambda \Lambda$;
    \item for all ordinals $\alpha, \beta < \Lambda$, it holds that $\alpha \leq \beta $ iff $\mathbb{N}\vDash \ulcorner\alpha \urcorner \leq_\Lambda \ulcorner\beta\urcorner$;
    \item $\ea \vdash \ ``\leq_\Lambda \textnormal{ is a partial order}"$.
\end{itemize}
Notice that we make no demands on $\leq_\Lambda $ being a well order or even linear. Since both $x <_\Lambda \Lambda$ and $x \leq _\Lambda y$ are $\Sigma_1$-formulas, we can use $\Sigma_1$-completeness to have for every representable theory $S \supseteq \ea$ that $\mathbb{N}\vDash x <_\Lambda \Lambda $ implies $\Box_S \dot{x} <_\Lambda \Lambda$, and similarly $\mathbb{N}\vDash x <_\Lambda y$ implies $\Box_S \dot{x} <_\Lambda \dot{y}$.
For the remainder of this paper we will write $\alpha < \beta $ instead of $ \ulcorner\alpha\urcorner <_\Lambda \ulcorner\beta \urcorner $ and $\alpha < \Lambda $ instead of $ \ulcorner\alpha\urcorner <_\Lambda \Lambda$.

With all that in mind, we can return to extending the base language with iterated truth predicates.
\begin{definition}
    Given an at most finite extension of the language of arithmetic $\mathcal{L}$, let $\mathcal{L}_\alpha := \mathcal{L} \cup \{ {\sf T}_\beta : \beta < \alpha  \}$. We then define $\utb_\alpha$ as the $\mathcal{L}_{\alpha+1}$ theory $\utb_{\mathcal{L}_\alpha}[ {\sf T } \leftarrow {\sf T}_{\alpha}  ]$.
    Additionally, we define:
    \[\utb_{<\alpha}:= \bigcup_{\beta<\alpha} \utb_\beta, \ \ \ \  \utb_{\leq \alpha} := \utb_{<\alpha} \cup\, \utb_\alpha.
    \]
    Given an ordinal $\alpha$, we write 
    \[\utb_{\floor*{\alpha}} := 
    \begin{cases}
        \utb_\beta, &\text{ if } \alpha = \omega(1+\beta) + n;\\
        \emptyset, &\text{ if } \alpha = n.
    \end{cases}
    \]
    Observe that the $\beta$ above is unique for given $\alpha$.
\end{definition}

 Typically, the language $\mathcal{L}_\alpha$ is going to be infinite. So in order to make use of Theorem \ref{thm:finite axiomatizability of rfn}, we will use a translation of formulas to a finite fragment of the language as is done in \cite{BeklemishevPakhomov:2019:GLPforTheoriesOfTruth}. Given an $\mathcal{L}$-formula $\varphi$ and some ordinal $\alpha$, let $\varphi^\bullet$ denote the result of the simultaneous substitution of ${\sf T}_\alpha({\sf T}_\beta(\dot{t})) $ for ${\sf T}_\beta(t)$ in $\varphi$ for every $\beta< \alpha$ (not substituting inside the terms $t$). Then we write $\utb_{\leq \alpha}^\bullet$ to denote the $\mathcal{L}_{\alpha+1}$-theory axiomatized by $\{ \varphi^\bullet : \varphi \in \utb_{\leq \alpha} \}$.

\begin{lemma}\label{lem: UTB<a in a finite language}
    For all $\varphi \in \mathcal{L}_{\alpha+1}$,
    \begin{itemize}
        \item $\ea + \utb_\alpha \vdash \varphi \leftrightarrow \varphi ^\bullet$;
        \item $\ea + \utb_{\leq \alpha} \vdash \varphi$ iff $\ea +\utb_{\leq \alpha} ^\bullet \vdash \varphi^\bullet $
    \end{itemize}
\end{lemma}
It is formalizable in $\ea$ that for any c.e. $\mathcal{L}$-theory $S\supseteq \ea$, the theory $S+\utb$ is a conservative extension over $S$ for $\mathcal{L}$-formulas \cite{Halbach:2014:AxiomaticTheoriesTruth}. In particular, given $\alpha < \beta$ and $S\supseteq \ea + \utb_{<\alpha}$ a c.e. $\mathcal{L}_\alpha$-theory, then $S+\utb_{<\beta}$ is a conservative extension over $S$ for $\mathcal{L}_{\alpha}$-formulas \cite{BeklemishevPakhomov:2019:GLPforTheoriesOfTruth}. 

From here on, we will assume that $\mathcal{L}\supseteq\mathcal{L}_{\pa}$ is at most a finite extension of the language of arithmetic. For a given elementary well-ordering $ (\Lambda, <) $, we expand it into an ordering of $ {(\omega(1+\Lambda),<)} $ by encoding $ \omega\alpha +n $ as pairs $ \la \alpha,n \ra $ with the expected ordering on them.
\begin{definition}[Hyperarithmetical hierarchy]
	For ordinals up to $ \omega(1+\Lambda) $, we define the hyperarithmetical hierarchy as ($\Sigma_\alpha$ is defined similarly):
	\begin{itemize}
		\item $\Pi_n := \Pi_n ^{\mathcal{L}} $, for every $ n<\omega $;
		\item $ \Pi_{\omega(1+\alpha)+n} := \Pi_{n+1}^{\mathcal{L}_\alpha}({\sf T}_\alpha) $;
		\item For $ \lambda $ a limit ordinal, we denote $ \Pi_{<\lambda}:= \bigcup_{\alpha<\lambda} \Pi_\alpha $.
	\end{itemize}
	
\end{definition}
For any theory $ S $ and for every $ \alpha, \lambda < \omega(1 + \Lambda) $, where $ \lambda $ is a limit ordinal, we define $R_\alpha (S) := \Pi_{1+\alpha}\textnormal{-}\rfn(S)$ and $R_{<\lambda} (S ) := \Pi_{<\lambda}\textnormal{-}\rfn(S)$.
Using Lemma \ref{lem: UTB<a in a finite language} and Theorem \ref{thm:finite axiomatizability of rfn}, we obtain:

\newpage
\begin{proposition}[Proposition 5.4 \cite{BeklemishevPakhomov:2019:GLPforTheoriesOfTruth}] \label{prop: UTB interraction with reflection}\
	\begin{enumerate}[label=(\roman*)] 
		\item If $ S \supseteq \ea + \utb_{\alpha} $, then over $ \ea + \utb _\alpha $,
		\[R_{\omega(1+\alpha)+n}(S) \equiv \Pi_{n+1} ^{\mathcal{L}} ({\sf T}_\alpha) \textnormal{-}\rfn(S);
		\]
		\item If $ S \supseteq \ea + \utb_{\alpha} $ and $ \beta = \omega(1+ \alpha) +n $, then $ R_\beta(S) $ is finitely axiomatizable over $ \ea + \utb_\alpha $;
		\item If $ S \supseteq \ea + \utb_{<\alpha} $, then over $ \ea + \utb _{<\alpha} $,
		\[ R_{<\omega(1+\alpha)}(S)\equiv \mathcal{L}_\alpha\mbox{-}\rfn(S)\equiv \{R_\beta(S): \beta < \omega(1+\alpha) \}.
		\]
	\end{enumerate}
\end{proposition}
Now we can define the interpretation of $[\alpha]\varphi$ that we will be using for the soundness proof.
\begin{definition}
    We will write $[\alpha]_S \varphi$ as a shorthand for the finite axiomatization of $\neg R_\alpha (S+ \neg \varphi)$ given by Statement (ii) of Proposition \ref{prop: UTB interraction with reflection}. which for $\alpha = \omega (1+\beta) +n$, is the $\Sigma_\alpha$-formula: 
    \[i\delta^{\mathcal{L}({\sf T}_\beta)} \to \exists \, \theta \in \Sigma_{n+1} ^{\mathcal{L}({\sf T}_\beta)} \big({\sf Tr}_{\Sigma_{n+1} ^{\mathcal{L}({\sf T}_\beta)}} (\theta) \wedge \Box_S (\theta \to \varphi)\big) ,
    \]
    where $i\delta^{\mathcal{L}({\sf T}_\beta)}$ is a finite $\Pi_1 ^{\mathcal{L}({\sf T}_\beta)}$-axiomatization of $I\Delta_0 ^{\mathcal{L}({\sf T}_\beta)}$.\\
    Similarly, by $\la \alpha \ra \varphi$ we will denote the finite axiomatization of $R_\alpha (S+ \varphi)$.
\end{definition}

An arithmetical realization is a function $(\cdot)^* _S$ from the language of $\glp_\Lambda$ to $\mathcal{L}_\Lambda$, mapping propositional variables to formulas of $\mathcal{L}_\Lambda$ and preserving the logical operations:
$(\varphi \wedge \psi)^*_S := \varphi_S^* \wedge \psi_S ^*$, $(\neg \varphi)^*_S := \neg \varphi_S^*$, and $([\alpha] \varphi)^*_S := [\alpha]_S \varphi_S ^*$. $\glp_\Lambda $ is sound for this interpretation.

\begin{theorem}\label{Arithmetical soundness for GLP L}
    For every $S \supseteq \ea + \utb_{<\Lambda}$ and every formula $\varphi$ in the language of $\glp_\Lambda$,
    \[\glp_\Lambda \vdash \varphi \Rightarrow \ea + \utb_{<\Lambda} \vdash (\varphi)^*_S, \text{ for every realization } (\cdot)^*_S \text{ of the variables of } \varphi.
    \]
\end{theorem}
The proof of soundness from here on is routine, starting with the corresponding provable completeness.
\begin{lemma}[Provable $\Sigma_\alpha$-completeness] \label{lem: provable Sigma b completeness}
    \[ \ea + \utb_{\floor*{\alpha}} \vdash \varphi \to [\alpha]_S \varphi \text{, if } \varphi \in \Sigma_{\alpha}.
    \]
\end{lemma}

\begin{proof}
    From Statement (ii) of Proposition \ref{prop: UTB interraction with reflection}, $ [\alpha]_S \varphi $ is finitely axiomatizable in $\ea + \utb_{\floor*{\alpha}}$. We will prove the contrapositive by reasoning within $\ea + \utb_{\floor*{\alpha}}$. Assume the finite axiomatization of $ R_\alpha(S+\neg\varphi) $, which implies $\Box_{S+\neg \varphi} \neg \varphi \to \neg \varphi$ because $\neg \varphi \in \Pi_\alpha$. Since $\Box_{S+\neg \varphi} \neg \varphi $ holds, $\neg \varphi$ follows.
\end{proof}
Now we have all the tools to prove L\"ob's derivability conditions:
\begin{lemma} 
    Let $\alpha < \beta$ and $\ea + \utb_{\floor*{\alpha}} +\utb_{\floor*{\beta}} \subseteq S$, then
    \begin{enumerate}[label=(\roman*)]
        \item If $S\vdash \varphi $ then $\ea + \utb_{\leq \Lambda} \vdash [\alpha]_S \varphi$;
        \item $\ea + \utb_{\leq \Lambda} \vdash [\alpha]_S (\varphi \to \psi )\to ( [\alpha]_S \varphi  \to [\alpha]_S \psi )$;
        \item $\ea + \utb_{\leq \Lambda} \vdash [\alpha]_S \varphi \to [\alpha]_S [\alpha]_S \varphi $;
        \item $\ea + \utb_{\leq \Lambda} \vdash [\alpha]_S \varphi \to [\beta]_S \varphi $;
        \item $\ea + \utb_{\leq \Lambda} \vdash \la \alpha \ra_S \varphi \to [\beta]_S\big( \la \alpha \ra _S \varphi \big) $.
    \end{enumerate}
\end{lemma}

\begin{proof}
    By statement (ii) of Proposition \ref{prop: UTB interraction with reflection} the $[\alpha]_S \varphi$ and $[\beta]_S\varphi$ formulas are well defined as the finite axiomatizations of $\neg R_\alpha (S+ \neg \varphi)$ and $\neg R_\beta (S+ \neg \varphi)$ respectively.
    \begin{enumerate}[label=(\roman*)]
        \item The assumption implies $\ea \vdash \Box_S \varphi$ and so statement $(i)$ follows.
        \item Immediate from the statement (i) of Proposition \ref{prop: UTB interraction with reflection} and Lemma \ref{lem: D_0 induction on finite language gives distributivity}. 
        \item Follows from Lemma \ref{lem: provable Sigma b completeness} as $[\alpha]_S \varphi$ is a $\Sigma_\alpha$ formula. 
        \item Assume that $\alpha= \omega \gamma + n$ and $\beta = \omega \delta + m$ with $\gamma < \delta$ and reasoning in $\ea + \utb_{\floor*{\alpha}} + \utb_{\floor*{\beta}}$ we remark that if a formula $\varphi$ is $\Sigma_\alpha$ then it is equivalent to ${\sf T}_\delta(\varphi)$ which is a $\Sigma_{\omega\delta}$-formula.
        \item Since $\la \alpha \ra_S \varphi$ is a $\Pi_\alpha$-formula then, reasoning as above, it is also a $\Sigma_\beta$-formula over\\ $\ea + \utb_\gamma + \utb_\delta$.
    \end{enumerate}
\end{proof}

\begin{lemma}[\cite{FernandezJoosten:2018:OmegaRuleInterpretationGLP}]
    Let $\bgl_\blacksquare$ be the extension of $\bgl$ by a new modal operator $\blacksquare$ and the axioms $\Box \varphi \to \blacksquare \varphi $, $\blacksquare \varphi \to \blacksquare\blacksquare \varphi $, and $ \blacksquare(\varphi \to \psi) \to (\blacksquare\varphi \to \blacksquare\psi) $.
 Then for all $\varphi$, $ \blacksquare(\blacksquare\varphi \to \varphi) \to \blacksquare\varphi$.
    \end{lemma}

Since L\"ob's theorem holds for $[0]$ as usual from the fixed point theorem, we conclude that it holds for all modalities, concluding our proof of Theorem \ref{Arithmetical soundness for GLP L}.

\begin{lemma}
    Let $\ea + \utb_{\floor*{\alpha}} \subseteq S$. Then, $ \ea + \utb_{\leq \Lambda } \vdash [\alpha]_S([\alpha]_S \varphi \to \varphi) \to [\alpha]_S\varphi$.
\end{lemma}

\section{Worm Battles beyond $\pa$}

Let $ \equiv_\alpha $ and $ \equiv_{<\lambda} $ denote equivalence for $ \Pi_{1+\alpha} $ and $ \Pi_{<\lambda} $-sentences respectively. In \cite{BeklemishevPakhomov:2019:GLPforTheoriesOfTruth} two conservation results are proven to hold provably in $\ea^+$: Theorems \ref{thm: generalized lim reduction property} and \ref{thm: generalized succ reduction property}. We fix a particular $\Lambda$.

\subsection{The Reduction Property}
The first conservation result centers around the case for reflection on limit ordinals.
\begin{theorem}\label{thm: generalized lim reduction property}
	Let $ \lambda {=} \omega (1+\alpha) $ and $ S {\supseteq} \ea + \utb_\alpha $. Over $ \ea + \utb_{<\Lambda} $,  $ R_{\lambda}(S) \equiv_{<\lambda} R_{<\lambda}(S) $.
\end{theorem}
The second concervation result centers around successors. It can be viewed as an extension of the so-called reduction property (cf. \cite{Beklemishev:2005:Survey}) to cover all successor ordinals and not just the finite ones.
\begin{theorem}\label{thm: generalized succ reduction property}
	Let $ V $ be a $ \Pi_{1+\alpha + 1} $-axiomatized extension of $ \ea + \utb_{<\Lambda} $ and let provably $ S \supseteq V $. Then, over $ V $, $ R_{\alpha + 1} (S) \equiv_\alpha \{ R_\alpha (S), R_\alpha (S + R_\alpha (S)), \ldots \} $.
\end{theorem}
As in the case of $\glp_\omega$ we can recast these conservation results in terms of our interpreted modalities (using the same notation for the modality and its arithmetical denotation).
\begin{corollary}[Reduction Property]
    If $ \beta \leq \alpha, \lambda < \Lambda$ with $\lambda$ being a limit ordinal, then
	\begin{align*}
	     \ea^+ + \utb_{<\Lambda} \vdash \la \beta \ra \la \alpha{+}1 \ra \varphi \leftrightarrow \forall \, k \ \la \beta \ra Q^\alpha _k (\varphi);\\
	     \ea + \utb_{<\Lambda} \vdash \la \beta \ra \la \lambda \ra \varphi \leftrightarrow \forall \, k \ \la \beta \ra \la \lambda[k] \ra \varphi.
	\end{align*}
\end{corollary}
\begin{toappendix}

\begin{proof}
	By Theorem \ref{thm: generalized succ reduction property} for $V = \ea + \utb_{\Lambda}$ and $S =V + \varphi$, we have
	\[ \{\la \alpha+1 \ra  \varphi\} \equiv_\beta \{ Q^\alpha _k (\varphi) : k<\omega \}
	\]
	holds over $\ea + \utb_\Lambda$ and is an equivalence formalizable in $\ea^+ + \utb_{<\Lambda}$. So over $ \ea^+ + \utb_{<\Lambda} $, a $ \Pi_{1+\beta} $ sentence is provable from $ \la \alpha+1 \ra \varphi  $ if and only if it is so from $ Q^\alpha _k (\varphi) $, for some $ k $ which proves the first Reduction Property.

	For the second, by Theorem \ref{thm: generalized lim reduction property} for $S = \ea + \utb_{\floor*{\lambda}} + \varphi$, over $\ea + \utb_{<\Lambda}$,
	$ \{ \la \lambda \ra \varphi \} \equiv_{<\lambda} \{ \la \gamma \ra \varphi: \gamma< \lambda\}
	$,
	and the equivalence is also provable over $\ea + \utb_{<\Lambda}$. So over $\ea + \utb_{<\Lambda}$, a $ \Pi_{1+\beta} $-sentence $\psi$ is provable from $ \la \lambda \ra \varphi  $ if and only if it is so from $\la \gamma \ra \varphi$ for some $\gamma < \lambda$. Let $k$ be such that $\gamma < \lambda[ k] $, then $\psi$ is also provable from $\la \lambda[ k] \ra \varphi$.
\end{proof}

\end{toappendix}

In the first order language $ \mathcal{L}({\sf T}) $, consider the following theory 
\[ \pa({\sf T}) := \ea {+} \utb_{\mathcal{L}({\sf T})} \ + \ R_{<\omega 2} ( \ea {+} \utb_{\mathcal{L}({\sf T})}), \] 
equivalent (provably so in $\ea^+$) to the corresponding $\pa({\sf T})$ in \cite{BeklemishevPakhomov:2019:GLPforTheoriesOfTruth} and to $\sf CT$ in \cite{Halbach:2014:AxiomaticTheoriesTruth}. We have the following well known result from \cite{Halbach:2014:AxiomaticTheoriesTruth}:
\begin{theorem}
	$ \pa({\sf T}) $ and $ \aca $ are proof theoretically equivalent. 
\end{theorem}
As such, we are going to use $\pa({\sf T})$ as a substitute for $\aca$ in our theorem on the equivalence between it and the corresponding worm principle. 

\begin{theorem} \label{thm:EWD omega2 equiv with 1-Con(ACA)}
	$ \ewd^{\omega 2} $ is equivalent to $ 1 $-$ \con(\pa ({\sf T}) ) $ in $ \ea $.
\end{theorem}

At the same time, we will prove the corresponding equivalence between the worm principle and $\is_n$. We will find however that we will be in need of a different method to prove the full equivalence.

\begin{theorem} \label{thm: EWD and ISn}
$\ewd^{n+1}$ is equivalent to $1\text{-}\con(\is_n)$ in $\ea$.
\end{theorem}

We will begin proving the theorems simultaneously since the corresponding proofs for both are similar. To this end, we will make an abuse of notation using the fact that provably over $\ea$, $\ea + \utb$ is a conservative extension of $\ea$ for $\mathcal{L}$-formulas.
For the remainder of this paper we write $ [\alpha] \varphi $ to mean $ [\alpha]_{\ea + \utb} \varphi $.
Note that if $\varphi \in \mathcal{L}$, then $ [\alpha] \varphi $ is equivalent to $ [\alpha]_{\ea} \varphi $ due to conservativity.
We will make the same convention for the $\la \alpha \ra \varphi$ and the proof theoretic worms.

\subsection{From 1-consistency to the worm principle}
The initial proof for both directions will follow the structure of the corresponding proof in \cite{Beklemishev:2005:Survey}. Observe the weaker implication in the case of $1 \text{-} \con(\is_n)$.
\begin{proposition} \label{Theorem:ACAIsigOneConsistencyImpliesEWD}\
	\begin{enumerate}
	    \item $ \ea +  1 \textnormal{-} \con(\pa ({\sf T}) ) \vdash \ewd^{\omega 2} $;
	    \item $\ea + 1 \text{-} \con(\is_n) \vdash \ewd^{n}$.
	\end{enumerate}
\end{proposition}

There is a distinction in the first step of this proof, due to the fact that $\is_n$ is an extension of $\ea$ with reflection for a successor ordinal, in comparison to $\pa$ or $\pa({\sf T})$, which correspond to reflection for a limit.
\begin{lemma}
	\label{lem: worms that PA(T) proves}
	For any $ A {\in} \w^{\omega 2} $, $ \pa({\sf T}) \vdash A $.
\end{lemma}
\begin{proof}
	For every $ A \in \w^{\omega 2} $, there is some $ m>0 $ such that $ A \in \w^{\omega {+} m} $ and so by Proposition \ref{prop: <a>^n+1 implies smaller A with <a>^n },
	$$ \glp \vdash \la \omega + m \ra\top \rightarrow A.
	$$
	Therefore, by arithmetical soundness of $\glp $, it holds that $ \ea + \utb \vdash \la \omega {+} m \ra\top \rightarrow A $ and since $ \pa({\sf T}) \vdash \la \omega{+}m \ra \top $, the lemma follows and its proof is formalizable in $ \ea $ (or $\ea^+$ if we are to use the corresponding $\pa({\sf T})$ in \cite{BeklemishevPakhomov:2019:GLPforTheoriesOfTruth}).
\end{proof}
Similarly for the $\is_n$, we have the corresponding theorem giving us the proof theoretic worms we can make use of in its case.
\begin{lemma}
	\label{lem: worms that IS proves}
	For any $ A {\in} \w^{n+1} $, $ \is_n \vdash A $.
\end{lemma}
\begin{proof}
	By Proposition \ref{prop: <a>^n+1 implies smaller A with <a>^n } we have that for every $ A \in \w^{n+1} $,
	$ \glp \vdash \la n+1 \ra\top \rightarrow A.
	$
	Therefore, by arithmetical soundness of $\glp $, it holds $ \ea \vdash \la n+1 \ra\top \rightarrow A $ and since $ \is_n \vdash \la n+1 \ra \top $, the lemma follows and its proof is formalizable in $ \ea^+ $.
\end{proof}

Now we introduce a notation we will use for the remainder of the proof of this direction. Given a worm $A$, we define $A^+$ inductively by $\top ^+ := \top$ and if $A = \la\alpha\ra B$ then $A^+ = \la \alpha +1\ra (B^+)$. 

\begin{lemma}
	\label{lem:provable reduction by next for ACA}
 
	$ \ea  \vdash \forall \, A {\in} \w^{\omega 2}\,  \forall \, k \ \big( A_k {\neq} \monid \rightarrow \Box (A_k ^+ \rightarrow \left<1 \right> A_{k+1 } ^+)\big)$. 
\end{lemma}

\begin{toappendix}

\begin{proof}
	It is sufficient to prove in $ \ea  $
	\[
	\forall\, A {\neq} \monid \ \forall \, k \ \ea + \utb \vdash A^+ \rightarrow \left<1 \right> A\llbracket k \rrbracket^+  .
	\]

	For this, we will move over to $ \glp_{\omega 2} $ where we have that the following proof is bounded by a function elementary in $A$ and $k$ and hence it is formalizable in $ \ea $ that $\glp_{\omega 2} \vdash A \rightarrow \Diamond A \llbracket k \rrbracket,$
	and as theorems of $ \glp_{\omega 2} $ are stable under the $(\cdot)^+$ operator, $\glp_{\omega 2} \vdash A ^+ \rightarrow \left< 1 \right> A \llbracket k \rrbracket ^+ ,$
	which by the arithmetical soundness of $ \glp_{\omega 2} $, proves that for every $ A \in \w^{\omega 2} $ with $ A \neq \monid $ and for every $ k $, 
	\[\ea + \utb \vdash A^+ \rightarrow \left<1 \right> A\llbracket k \rrbracket^+  .
	\]
	From here, we are of course unable to use $ \Sigma_1 $-induction to prove 
	\[
	\ea \vdash \forall \, k \ \big( A_k {\neq} \monid \rightarrow \Box (A_k ^+ \rightarrow \left<1 \right> A_{k+1 } ^+)\big), 
	\]
	which is how we would --in principle-- expect to complete the proof. Instead we use the fact that for a given $ k $,
	the proof of $A_k ^+ \rightarrow \left<1 \right> A_{k+1 } ^+$ is bounded by an elementary function of $A$ and $ k $. The proof itself can be formalized within $ \ea $ and therefore the formula $ \Box (A_k ^+ \rightarrow \left<1 \right> A_{k+1 } ^+)\big) $ can be written as a $ \Delta_0 $-formula by placing the existential quantifier inside this bound. So we complete the proof with a $ \Delta_0 $-induction.
\end{proof}

\end{toappendix}

\begin{lemma}
	\label{lem: 1 con of worm proves it dies up to omega 2}
	$ \ea  \vdash \forall A {\in} \w ^{\omega 2}  \big( \left< 1 \right> A^+ _0 \rightarrow \exists \, m \ A_m {=} \monid \big) .$
\end{lemma}

\begin{toappendix}

\begin{proof}
	We prove the contrapositive. 
	The first part of our reasoning will prepare for an application of L\"ob's theorem.
	Reasoning within $\ea$,
	\[  [1] \, \forall \, m \ [1] \neg A_m ^+ \vdash [1] \, \forall \, m \ [1] \neg A_{m+1} ^+
	\vdash  \forall \,m \ [1][1] \neg A_{m+1} ^+
.\]
Therefore, using Lemma \ref{lem:provable reduction by next for ACA} in the form $ \ea  \vdash \forall \, k \ \big( A_k {\neq} \monid \rightarrow [1]  ( [1] \neg A_{k+1} ^+ \rightarrow \neg A_{k } ^+)\big)$,
		\[\forall \, m \ A_m {\neq} \monid \ \wedge \ [1] \, \forall \, m \ [1] \neg A_m ^+ \vdash 
		\forall \, m \ A_m {\neq} \monid \ \wedge \ \forall \,m \ [1][1] \neg A_{m+1} ^+
	\vdash \forall \, m \ [1] \neg A_m ^+ .\]
Thus $\ea   \vdash \forall \, m \ A_m {\neq} \monid \rightarrow ([1] \, \forall \, m \ [1] \neg A_m ^+  \rightarrow \forall \, m \ [1] \neg A_m ^+).$
	Then, after necessitation on the $[1]$-modality and distribution we have
$\ea  \vdash [1] \, \forall \, m \ A_m {\neq} \monid \rightarrow [1]([1] \, \forall \, m \ [1] \neg A_m ^+  \rightarrow \forall \, m \ [1] \neg A_m ^+),$
hence by Löb's theorem		$\ea  \vdash [1] \, \forall \, m \ A_m {\neq} \monid \rightarrow \, [1] \forall \, m [1] \neg A_m ^+$.

	Now observe that $ \forall \, m \ A_m {\neq} \monid $ is $ \Pi_1 $, so certainly $ \Sigma_2 $ and hence by $\Sigma_2$-completeness
\[\ea  \vdash \forall \, m \ A_m {\neq} \monid  \rightarrow [1] \, \forall \, m \ A_m {\neq} \monid.\]
But then in $\ea$,
\[\forall \, m \ A_m {\neq} \monid  \vdash \forall \, m \ A_m {\neq} \monid \ \wedge \ [1] \, \forall \, m [1] \neg A_m ^+ \vdash   \forall \, m \ [1] \neg A_m ^+\vdash [1] \neg A_0 ^+.\]
	By contraposition $\ea\vdash \langle 1\rangle A_0 ^+ \to \exists \, m \ A_m {=} \monid$, as desired.
\end{proof}

\end{toappendix}

Note that the use of $A^+$ in the above lemma does not allow us to apply it to $\ewd^{n+1}$ in place of $\ewd^{\omega 2}$.
Moreover, it cannot be avoided using the current proof.
Now we prove Proposition \ref{Theorem:ACAIsigOneConsistencyImpliesEWD}: from Lemmata \ref{lem: worms that PA(T) proves} and \ref{lem: 1 con of worm proves it dies up to omega 2} we obtain that for each $ A {\in} \w^{\omega 2} $, $\pa({\sf T})  \vdash \left< 1 \right> A^+$ and over $\ea$
\[
\left< 1 \right>\pa({\sf T}) \vdash 
\forall A {\in} \w^{\omega 2} \left< 1 \right> A^+ \vdash 
\forall \, A {\in} \w^{\omega2} \ \exists \, m \ A_m {=} \monid \vdash \ewd^{\omega 2} .
\]


Similarly for the case of $\is_n$, from Lemmata \ref{lem: worms that IS proves} and \ref{lem: 1 con of worm proves it dies up to omega 2} we obtain that formalisably in $\ea$, for each $ A \in \w^{n} $, $\is_n   \vdash \left< 1 \right> A^+$ and $\ea  \vdash \left< 1 \right> A^+ \rightarrow \exists \, m \ A_m {=} \monid$.
Hence, as before we obtain $ \forall \, A {\in} \w^{n} \ \exists \, m \ A_m {=} \monid $, which is $ \ewd^{n} $.

\subsection{From the worm principle to 1-consistency}

Now we prove the second direction of Theorem \ref{thm:EWD omega2 equiv with 1-Con(ACA)}, proving independence of $ \ewd^{\omega 2} $.
\begin{proposition} \label{prop:EWD to ACA}\
    \begin{enumerate}
        \item $ \ea +  \ewd^{\omega 2} \vdash   1 \textnormal{-} \con(\pa ({\sf T}) ) $;
        \item $\ea + \ewd^{n+1} \vdash 1\text{-}\con(\is_n)$.
    \end{enumerate}
	
\end{proposition}

We use a Hardy functions' analogue on worms $ h_A(m) $ defined as the smallest $ k $ such that $ A \llbracket \frto mk \rrbracket {=} \monid $, where $ A \llbracket \frto mk \rrbracket {:=} A \llbracket m\rrbracket \ldots \llbracket m{+}k \rrbracket $. Each function $h_A$ is computable and hence there is a natural $\Sigma_1$ presentation of $ h_A(m)=k $ in $ \ea $. We will use the following relation to prove monotonicity for the $ h_A $ function.

\begin{definition}
	For $ A,B \in \w^{\omega 2} $, we define the partial ordering $ B \trianglelefteq A $ iff $ B= \monid $ or $A = D \alpha C$ and $ B = \beta C$ for some $\beta \leq \alpha$. \\ For every natural number $ m $, we define $ B \trianglelefteq_m A $ iff $ B \trianglelefteq A $ and additionally, if $ B= n C $  with $n<\omega$ and, $ A = D \alpha C $ with $\alpha \geq \omega$, then $ n \leq m $.
\end{definition}

Of course, by the definition, we immediately have that if $  B \trianglelefteq_m A $ and $ m \leq n $ then $  B \trianglelefteq_n A $. Additionally, if $ A=CB $ for some $ C $ then $ B \trianglelefteq_{m} A $ for every $ m\geq 0 $. Over $ \ea $, and for worms in $ \w^{\omega 2} $, we have the following:

\begin{lemma} \label{lem: B <| A and h_A halts then we can approach B for ACA}
	If $ h_A (m) $ is defined and $ B\trianglelefteq_{m} A $, then 
	$
	\exists \, k \ A\llbracket \frto mk \rrbracket {=} B. $
\end{lemma}

\begin{toappendix}
\begin{proof}
    
	The Definition of the step-down function $A\llbracket \cdot \rrbracket$ is such that an ordinal $ \alpha_i $ of $ A = \alpha_{\left| A\right| {-}1} \ldots \alpha_0 $ can only change if all elements to the left of it are deleted. So by the assumption of $ A\llbracket \frto ms \rrbracket {=} \monid $ there is some $k_0$ such that $A\llbracket \frto m{k_0} \rrbracket= \alpha_{\left|B\right| {-}1} \ldots \alpha_0$. We consider the case where $\alpha _{ \left| B \right| {-}1} \geq \omega $ and the corresponding ordinal $\beta_{ \left| B \right| {-}1} $ in $ B = \beta_{ \left| B \right| {-}1} \ldots \beta_0$ is $ {<} \omega $; the other cases are similar. Then by assumption of $ B\trianglelefteq_{m} A $, the ordinal $\beta_{ \left| B \right| {-}1} $ is also some $  n{\leq} m $.
	
	Let $ \alpha_{\left|B \right|{-}1 } = \omega + l $, then we can prove with $\Delta_0$-induction on $l$ bounded by $s$ that there is some $ k_1 $ such that $ A \llbracket \frto m{k_1} \rrbracket = \la \omega \ra C $ where $ B = \la n \ra C $ and $ A = D \la \omega + l \ra C $. Then $ A \llbracket \frto m{k_1 + 1} \rrbracket = \la m + k_1 + 1 \ra C $. With a second $\Delta_0$-induction bounded by $s$, we can find as before some $ k<s $ such that $ A\llbracket \frto mk \rrbracket =B  $.\\
	A more detailed proof can be found in the proofs of lemmata 9.4.3 and 6.3.3 in \cite{Papafillipou:2020:MastersThesis}.
\end{proof}

\end{toappendix}
The above can be easily expanded into the following:
\begin{corollary} \label{cor: h goes through initial segments in ACA}
	If $ h_A(n) $ is defined and $ B \trianglelefteq_n A $, then $ \forall \, m {\leq}n \ \exists \, k \ A \llbracket \frto nk \rrbracket {=} B\llbracket m \rrbracket. $
\end{corollary}
\begin{proof}
	By Lemma \ref{lem: B <| A and h_A halts then we can approach B for ACA}, there is $ k_1 $ such that $ A \llbracket \frto n{k_1} \rrbracket {=} B $. Then, as $k_1+1 > n \geq m$, there is some $C$ such that $ A \llbracket \frto n{k_1 {+}1} \rrbracket = C B\llbracket m \rrbracket $ and since $ h_A(n) $ halts, we can use Lemma \ref{lem: B <| A and h_A halts then we can approach B for ACA} once more to show that there is some $ k_2 $ such that $ A \llbracket \frto n {k_1{+}k_2} \rrbracket {=} B\llbracket m \rrbracket $.
\end{proof}
Using this result, we have the following monotonicity statement:
\begin{lemma} \label{lem: h is order preserving over the new order and natural ordering for ACA}
	If $ h_A(y) $ is defined, $ B \trianglelefteq_y A $ and $ x \leq y $, then $ h_B(x) $ is defined and $ h_B (x) \leq h_A(y). $
\end{lemma}
\begin{proof}
	By applying Corollary \ref{cor: h goes through initial segments in ACA} several times, we obtain $ s_0,s_1,\ldots $ such that
 $A\llbracket \frto y {s_0} \rrbracket  =  B \llbracket x \rrbracket$,  where $ y + s_0 \geq x$,
	$A\llbracket \frto  y{s_0 +s_1} \rrbracket  =  B \llbracket x \rrbracket \llbracket x+1 \rrbracket $,  where $y + s_0 +s_1 \geq x+1$, etc.
	Hence all elements of the sequence starting with $ B $ occur in the sequence for $ A $ and since $ h_A(y) $ is defined, so is $ h_B(x) $.
\end{proof}

Next we look into some results that bound the functions $h_A$ from below and compare them with some fast growing functions.

\begin{lemma} \label{lem: h breaking B0A into functional composition}
	For every $ A, B \in \w^{\omega 2} $, if $ h_{B0A}(n) $ is defined, then
	\[ h_{B0A}(n) =  h_A\big (n+h_B(n) +2 \big ) + h_B(n) +1 > h_A \big (h_B (n) \big ). \]
\end{lemma}

\begin{proof}
	Since $ 0A \trianglelefteq_0 B0A $, by lemma \ref{lem: B <| A and h_A halts then we can approach B for ACA} we have that $ h_B (n) $ is defined. As $ B0A $ first rewrites itself to $ 0A $ in $ h_B(n) $ steps and then begins to rewrite $ A $ into $ \monid $ at step $ n+h_B(n) +2 $, we have that $ h_A (n+h_B(n) +2) $ is then defined. Finally, by Lemma \ref{lem: h is order preserving over the new order and natural ordering for ACA}, $ h_A(h_B(n)) \leq h_A (n+h_B(n) +2) $ and it is also defined.
\end{proof}
Seeing how easy it is to achieve a lower bound based on the composition of functions, we can proceed by trying to get in-series iterations of this. Since the $ h_A $ functions are in general strictly monotonous, we will be getting faster and faster growing functions by following this method.
\begin{corollary}\label{cor: h growth}
	If $ A \in \w_1 ^{\omega 2} $ and $ h_{1A} (n) $ is defined, then $ h_{1A} (n) > h_A ^{(n)}(n). $
\end{corollary}
\begin{proof}
	Since $ (1A)\llbracket n \rrbracket {=} (0A)^{n+1}       $, we can perform induction on the number of in-series concatenations of $ 0A $ by applying Lemma \ref{lem: h breaking B0A into functional composition}.
\end{proof}

As an application of this, we can see how quickly we reach superexponential growth.

\begin{corollary} \label{cor: h and superexponentiation}
	If $ h_{1111}(n) $ is defined then, $ h_{1111}(n) > 2^n _n $ and $ h_{111}(n) > 2^n $.
\end{corollary}
\begin{proof}
	We will make use of Corollary \ref{cor: h growth} multiple times. Clearly we first have that $ h_{1111}(n) > h^{(n)} _{111}(n) $, then $ h_{111} (n) > h^{(n)}_{11}(n) $ and $ h_{11}(n) > h^{(n)} _{1} (n) $. We can easily prove by induction in $ \ea $ that $ h_1 (n) = n+1 $. So by applying the compositions, $ h_{11}(n) > 2n $ and so $ h_{111}(n) > 2^n $ and finally $ h_{1111}(n) > 2^n _n $.
\end{proof}

At this point we find ourselves equipped to tackle the main lemma on which the proof of this direction rests. Due to the complexity added by the limit ordinal $\omega$, there is a technical addition in this proof when compared to the corresponding proof for $\pa$ in \cite{Beklemishev:2005:Survey}.
\begin{lemma} \label{lem:h haults into 1 con of worm for ACA}
	$ \ea  \vdash \forall \, A {\in} \w ^{\omega 2} _1 \ (h_{A1111}{\downarrow} \ \rightarrow \left< 1 \right> A). $
\end{lemma}
\begin{toappendix}

\begin{proof}
	By Löb's Theorem, this is equivalent to proving

\begin{equation}\label{equation:NiceLoebApplication}\ea  \vdash \Box \big( \forall \, A {\in} \w _1 ^{\omega 2} \ (h_{A1111}{\downarrow} \ \rightarrow \left< 1 \right> A) \big) \rightarrow  \forall \, A {\in} \w _1 ^{\omega 2} \ (h_{A1111}{\downarrow} \ \rightarrow \left< 1 \right> A).
\end{equation}
	We reason in $ \ea  $. Let us take the antecedent of \eqref{equation:NiceLoebApplication} as an additional assumption, which by the monotonicity axiom of $ \glp_{\omega } $ interpreted in $ \ea $, implies $ [1] \big( \forall \, A {\in} \w _1 ^{\omega 2} \ (h_{A1111}{\downarrow} \ \rightarrow \left< 1 \right> A) \big) $. This in turn implies:
	\begin{align} \label{assum: h halts to 1 consistency for ACA}
	\forall \, A {\in} \w^{\omega 2} _1 \ [1](h_{A1111 }{\downarrow} \ \rightarrow \left< 1 \right> A).
	\end{align}
	
	We make a case distinction on whether $A1111$ starts with a 1 or with an ordinal strictly larger than 1.\\ 
	If $ A1111 = 1B $ then by Corollary \ref{cor: h growth}, we have $ h_{1B} {\downarrow} \ \rightarrow \lambda x.h_B ^{(x)}(x) {\downarrow}. $	
	The function $ h_B $ is increasing, has an elementary graph and grows at least exponentially as per Corollary \ref{cor: h and superexponentiation}, $ h_{111} > 2^x $. So for $ A{=} \monid $ we have that $ h_{1111}{\downarrow} $ implies the totality of $ 2^x_n $ and hence $ \ea^+ $, which by Corollary \ref{cor: ea^+ is equiv to 1-con(ea)}, implies $ {\left< 1 \right>} \top $.\footnote{Since over $\ea$, it is provable that $\ea +\utb$ is conservative over $\ea$, by their definition, $\la 1 \ra_{\ea} \top$ and $\la 1 \ra_{\ea + \utb} \top$ are equivalent.} If $ A $ is nonempty, we reason as follows:
	\begin{align*}
	\lambda x . h_B ^{(x)} {\downarrow} & \vdash \left< 1 \right> h_B {\downarrow}, \ \ \ \text{ by Lemma \ref{lem: f^x (x) halts iff 1-con that f halts}}\\ 
	& \vdash \left< 1 \right> \left< 1 \right> B, \ \ \ \text{ by Assumption } \eqref{assum: h halts to 1 consistency for ACA}\\
	& \vdash \left< 1 \right> A. 
	\end{align*}
	If $ A1111 {=} C $ starts with $ \alpha > 1 $, we have $h_C {\downarrow}   \vdash \lambda x. h_{C \llbracket x \rrbracket } (x+1) {\downarrow} \vdash \forall \, n \ h_{C\llbracket n \rrbracket} {\downarrow}$.
	The last implication is derived by application of Lemma \ref{lem: h is order preserving over the new order and natural ordering for ACA} as for arbitrary $ n $, if $ x \leq n $ then $ h_{C\llbracket n \rrbracket }(x) \leq h_{C\llbracket n \rrbracket }(n+1) $ and if $ n \leq x $ then $ h_{C\llbracket n \rrbracket }(x) \leq h_{C\llbracket x \rrbracket }(x+1) $. In both cases, the larger value is defined.\\
	We can perform this line of argument a second time, something we will use for the case that $\alpha=\omega$, obtaining
\[\forall \, n \ h_{C\llbracket n \rrbracket} {\downarrow} \vdash \forall \, n \ \lambda x. h_{C \llbracket n \rrbracket \llbracket x \rrbracket } (x+1) {\downarrow} \vdash \forall \, n \ h_{C\llbracket n \rrbracket \llbracket n+1 \rrbracket} {\downarrow}.\]

	Now notice that no matter what the $ \alpha $ is, we will always have that either $ 1(C\llbracket n\rrbracket ) \trianglelefteq_1 C \llbracket n+1 \rrbracket $ or $ 1(C\llbracket n\rrbracket) \trianglelefteq_1 C \llbracket n+1 \rrbracket \llbracket n+2 \rrbracket $. To prove this, let $ D $ be such that $ C = \alpha D 1111 $. \\
	If $ \alpha = \omega $, then $ 1(C \llbracket n \rrbracket) = 1nD1111 $ and $ C\llbracket n{+}1 \rrbracket = \la n+1 \ra D1111 $ therefore $ C \llbracket n{+}1 \rrbracket \llbracket n{+}2 \rrbracket =  (n h_{n+1}(D1111))^{n+3} r_{n+1}(D1111) = (n h_{n+1}(D1111))^{n+2} n D1111 $. So if $ n = 0 $ then since $ D \in \w_1 ^{\omega 2} $, we have that $ r_{1}(D1111) = \monid $ and therefore,
	\[ C \llbracket n+1 \rrbracket \llbracket n+2 \rrbracket =   (0 D1111)^{0+2} 0 D1111 = 0 D11110 D111 10 D1111 = 0 D11110 D111 1 C \llbracket n \rrbracket .\]
	If $ n >0 $ then clearly $ n h_{n+1}(D1111) $ has as its rightmost element something $ \geq 1 $ and so $ 1(C\llbracket n\rrbracket )\trianglelefteq_1 C \llbracket n+1 \rrbracket \llbracket n+2 \rrbracket $. \\ 
	If $ \alpha \neq \omega $ then, $ 1 (C \llbracket n \rrbracket) = 1(\la \alpha -1 \ra h_{\alpha}(D1111))^{n+1} r_\alpha (D 1111) $ and
	\[  C \llbracket n{+}1 \rrbracket = (\la \alpha -1 \ra h_{\alpha}(D1111))^{n+2} r_\alpha (D 1111).\]
	So $ 1(C\llbracket n\rrbracket) \trianglelefteq_1 C \llbracket n+1 \rrbracket $.
	Therefore  we have:
	\begin{align*}
	\forall \, n \ h_{C} {\downarrow} & \vdash \forall \, n \ h_{1(C\llbracket n \rrbracket) } {\downarrow} \ \ \ \ (\text{by the above})\\
	& \vdash \forall \, n \ \lambda x . h_{C\llbracket n\rrbracket}^{(x)}(x) {\downarrow}\\
	& \vdash \forall \, n \ \left< 1 \right> h_{C\llbracket n\rrbracket} {\downarrow} \ \ \ \ (\text{by 	Lemma \ref{lem: f^x (x) halts iff 1-con that f halts}}).
	\end{align*}
	Again observe that since A starts with something bigger than 1, we have $C\llbracket n\rrbracket = A\llbracket n\rrbracket 1111$, hence we can apply our assumption. Hence the argument continues,
	\begin{align*}
	& \vdash \forall \, n \left< 1 \right>(\left< 1 \right> A \llbracket n\rrbracket ) \ \ \ \text{ by Assumption } \eqref{assum: h halts to 1 consistency for ACA}\\
	& \vdash \left< 1 \right> A \ \ \ (\text{by the reduction property}). 
	\end{align*}
	The last step is achieved because $ h_C {\downarrow} $ implies $ h_{1111} {\downarrow} $ which, as per our first step in this proof, implies $ \ea^+ $, hence allowing the use of the reduction property.
\end{proof}

\end{toappendix}

Now to prove Proposition \ref{prop:EWD to ACA} assume that $ \ewd^{\omega 2} $ holds. In $   \ea + \utb$ we have that
\[ \forall \, A {\in} \w^{\omega 2} \ \exists \, m \ A_m {=} \monid \vdash \forall \, A {\in} \w_{1}^{\omega 2} \ h_A {\downarrow}\vdash \forall \, n \ \left< 1 \right> {\left< \omega + n \right>} \top \\
    \vdash 1\text{-}\con(\pa({\sf T})).
\]
The first implication holds since for every worm $ A $ and every number $ x $, there is a worm $ A' = 0^x A $ where $ A'\llbracket \frto 0 {x-1} \rrbracket = A $ hence $ \exists \, m \ A'_m {=} \monid $ iff $ h_A(x) $ is defined.

As for the case of $\ewd^{n+1}$, assume that $ \ewd^{n+1} $ holds. We have in $\ea$ that
\begin{align*}
	\forall \, A {\in} \w^{n+1} \ \exists \, m \ A_m {=} \monid & \vdash \forall \, A {\in} \w_{1}^{n+1} \ h_A {\downarrow}\\
	& \vdash \forall \, k \ \left< 1 \right> \big({\left< n+1 \right>} \top \llbracket k \rrbracket\big), \ \ \text{ by Lemma \ref{lem:h haults into 1 con of worm for ACA}}\\
	& \vdash {\left< 1 \right> \left< n+1 \right>} \top \ \ \ \ (\text{by the reduction property})\\
	& \vdash 1\text{-}\con(\is_n).
\end{align*}
For the use of the reduction property, notice that here $ {\left< 1 \right> \left< n+1 \right>} \top \rightarrow {\left< 1 \right>} \top $ which in turn implies $ \ea^+ $.

\section{Worm battles below $\pa$}

In this section, we will prove that $\ea + 1\text{-}\con(\is_n) \vdash \ewd^{n+1}$ in Theorem \ref{theorem:fineStructureForWorms} below. In order to prove this, we will need to develop some technicalities involving so-called Hardy functions.

\subsection{Hardy functions}\label{subsect: Hardy functions}
The so-called \textit{Hardy functions} are functions on the natural numbers that are indexed by the ordinals. The collection of Hardy functions that are provably total within a theory reflect much of the proof-theoretical properties of that theory. Since worms (modulo provability) stand in one-one relation with the ordinals, we can also define Hardy functions indexed by worms. 

The main difference is that worms do not behave exactly the same way as ordinals as they may have additional structure and do not behave syntactically the same way as the normal forms of the corresponding ordinals. We will first inspect a somewhat simplified version of that behavior in the so-called \textit{tree ordinals} and compare their induced Hardy hierarchies to one commonly used on ordinals.
\begin{definition}[Tree ordinals]
    The set of tree ordinals $\mathcal{T}$ is the least set of terms defined as follows:
    \begin{itemize}
        \item $0\in \mathcal{T}$;
        \item Given $t_1, \ldots, t_n \in \mathcal{T}$ then $\omega^{t_1} + \ldots + \omega^{t_n} \in \mathcal{T}$.
    \end{itemize}
\end{definition}
For $t\in \mathcal{T}$ and $x\in \mathbb N$ we define $t\cdot x$ as $t\cdot 0 := 0$ and $t \cdot (x+1) := t\cdot x + t$. By $1$ we denote the tree ordinal $\omega^0$ and, given a natural number $s$ by $s$ we denote $1\cdot s$. The limit tree ordinals are those whose rightmost summand is $\omega^t$ with $t\neq 0$.
With this notation at hand we can define the usual \textit{fundamental sequences} on the tree ordinals as follows:

\begin{definition}
Let $t\in \mathcal{T}$ and $x \in \mathbb{N}$. Furthermore, let $\lambda$ be a limit ordinal. We define:

\begin{itemize}
    \item $0[ x ] := 0$;
    \item $1[ x ] := 0$;
    \item $(t+1) [ x ] := t $;
    \item $(t+ \omega^{s+1})[ x ] := t+ \omega^s \cdot x $;
    \item $(t+\omega ^\lambda) [ x ] := t + \omega ^{\lambda [ x ]}$.
\end{itemize}
We write $t\leq_n s$ to denote $t=s[n]^m$ for some natural number $m$.
\end{definition}

We will use the very same notation for tree terms as we do for base-$\omega$ Cantor normal forms of ordinals.
The induced ordinal of a term $t= \omega^{t_1} + \ldots + \omega^{t_n}$ is given by the function $o(t) = \omega^{o(t_1)} + \ldots + \omega^{o(t_n)}$ with $o(0) = 0$. Thus, the resulting CNF notation for $o(t)$ will be rather similar to $t$ where `smaller' terms may vanish.

Based on the tree notations, we can consider for each $t\in \mathcal T$ the corresponding the so-called \textit{Hardy function} $H_t(x) : \mathbb{N} \to \mathbb{N}$.

\begin{definition}
    For $x\in \mathbb{N}$ and $t\in \mathcal T$ we define
    \begin{itemize}
    \item $H_0 (x) = x$;
    \item $H_{t+1}(x) = H_t(x+1)$;
    \item $H_t (x) = H_{t[x]}(x)$, where $t$ is a limit term.
\end{itemize}
\end{definition}

We shall need to compare these functions to the regular Hardy functions that are indexed by ordinals. We recall the definition of those.

\begin{definition}
    For $x\in \mathbb{N}$, for $\alpha, \lambda < \varepsilon_0$ where $\lambda$ is a limit ordinal, we define
    \begin{itemize}
    \item $H_0 (x) = x$;
    \item $H_{\alpha+1}(x) = H_\alpha(x+1)$;
    \item $H_\lambda (x) = H_{\lambda[x]}(x)$, where $\lambda$ is a limit ordinal.
\end{itemize}

\end{definition}

Clearly, for every $\alpha < \varepsilon_0 $ there is a canonical term $t$ such that $o(t) = \alpha$ and $H_\alpha(x) = H_t (x)$. In general, we will have that $H_t(x) \geq H_{o(t)}(x)$. However, we can estimate that $H_t(x)$ does not grow essentially much faster than $H_{o(t)}(x)$. In particular, we claim that for every tree term $t$, there is a natural number $c_t$ so that for every number $x$ we have that $H_t (x) \leq H_{o(t)} (x+c_t)$. Moreover, this number $c_t$ is elementary definable from $t$.

The evaluation of the value of either $H_t(x)$ or $H_\alpha(x)$ functions may be recorded in {\em evaluation sequences}
\[e = \big ((\xi_0,x_0,0),\ldots,(\xi_i,x_i,i),\ldots,(\xi_n,x_n,n) \big),\]
where assuming that $e$ is an evaluation sequence for $H_t(x)$, then $(\xi_n,x_n,n)=(t,x,n)$, $\xi_0=0$ and for $i<n$ we have that $\xi_{i} = \xi_{i+1}[x_{i+1}]$ and $x_{i} = x_{i+1} + \delta$, where $\delta=0$ if $\xi_{i+1}$ is a limit and $1$ otherwise.
These sequences can be used to formalize Hardy hierarchies within arithmetic using $\Sigma_1$ formulas, and will be essential in our treatment within $\ea$ below.

A common tool for comparing Hardy functions is to define a norm function $ N: \mathcal{T} \to \mathbb{N}$ as a weak bounding measure on the structure of the tree ordinals.
\begin{definition}[Norm function]
    We consider the norm function $N$ defined on the ordinals $<\varepsilon_0$ as well as the tree ordinals, as follows:
    \begin{itemize}
        \item $N0 = 0$;
        \item $N(\omega ^\alpha + \beta) = 1 + N\alpha + N\beta$.
    \end{itemize}
    When $N$ is applied to ordinals below $\varepsilon_0$ we shall require that $\beta < \omega^{\alpha+1}$ in the second clause and we observe that such a decomposition is indeed unique.
\end{definition}

However, the norm functions alone are not sufficient to set up the inductive arguments needed for our comparison between the different Hardy functions. Therefore, the intent is to use a much weaker measure to dilute the progressive growth of the difference between the norm of a tree ordinal and its corresponding ordinal. 

We write $t<s$ to abbreviate $o(t)<o(s)$.
\begin{definition}
    Given a tree ordinal $t = \omega ^{t_1} + \ldots + \omega^{t_n} \in \mathcal{T}$, we inductively define its ordinal correction function via $Cr(0) = 0$ and, 
    \[
    Cr(t) = \sum_i \{ N(\omega ^{t_i}) : t_i < t_j \text{ for some } j>i\} + \max\{ Cr(t_i) : i \leq n\} . 
    \]
\end{definition}

The ordinal correction function of a tree ordinal intents to provide a measure on the terms lost as we transform the tree ordinal into its corresponding ordinal and compare their notations. An important property it must satisfy is the following:

\begin{lemma}[EA]
$Cr(t[x]) \leq Cr(t)$ for every tree ordinal $t$ and every natural number $x$. 
\end{lemma}
\begin{proof}
By induction on the structure of the tree ordinal. Let $t = \omega ^{t_1} + \ldots + \omega^{t_n} $, we note that if $t$ is a successor, we have equality and so we consider two cases:
\begin{enumerate}
    \item ($t_n = s+1$). Then
    \[t[x] = \omega ^{t_1} + \ldots + \omega^{t_{n-1}}+\omega^{s} x = \omega ^{t'_1} + \ldots + \omega^{t'_{n+x}},\]
    where
    \[t'_i = \begin{cases}
        t_i & \text{ if } i<n\\
        s & \text{ otherwise}.
    \end{cases}
    \]
    If $t'_i < t'_j $ then $i < n$ and if $j\geq n$ then $t_i = t'_i < t_n$. Additionally, $Cr(t_n) = Cr(s)$ and thus $Cr(t[x]) \leq Cr(t) $.
    \item ($t_n \in {\sf Lim}$). Then
    \[t[x] = \omega ^{t_1} + \ldots + \omega^{t_n[x]},\]
    by IH $Cr(t_n[x]) \leq Cr(t_n)$ and so $Cr(t[x]) \leq Cr(t)$.
\end{enumerate}
\end{proof}

The following lemma is folklore and will later be used to divide the main proof into two cases, thus restricting the structural distinctions to be considered between the ordinals and the terms.

\begin{lemma}[EA]\label{lemmNormBound}
If $\alpha<\beta$, if $x\geq 2$ and $N\alpha\leq N\beta+x-2$ then, $\alpha \leq \beta[x]$. If $\beta$ is a limit, the inequality is strict and $N\alpha \leq N\beta[x] + x -2$.

\end{lemma}

\begin{proof}
The case for $\beta$ a successor ordinal is immediate. We prove that $\alpha < \beta[x]$ and $ N\beta \leq N(\beta[x])$ by induction on norms,
\begin{enumerate}
    \item ($\beta = \gamma + \omega^{\delta +1 } $). If $\alpha \geq \beta[x]$, then $\alpha = \gamma  + \omega ^\delta x + \rho $ but then $N \alpha  = N \beta  - 1 + N( \omega ^\delta (x-1)) + N \rho > N \beta + x -2$, a contradiction. Additionally, $N\beta[x] = N (\gamma + \omega ^\delta x) = N\beta -1 + N \omega^ \delta (x-1) \geq N \beta $. 
    \item ($\beta = \gamma + \omega ^\lambda $). If $\alpha \geq \beta[x]$, then $\alpha = \gamma  + \omega ^{\lambda'} + \rho$ with $\lambda >\lambda' \geq \lambda[x]$ and $N \lambda' <  N\lambda + x -2 $, so by IH, $\lambda' < \lambda[x]$ as a limit ordinal, a contradiction. Additionally, by IH $N \lambda [x] \geq N \lambda $ and therefore, $N\beta[x] \geq N \beta $.
\end{enumerate}
\end{proof}

So we can direct our attention at the case where $o(t) =\beta$ and the comparison is then to be made between the application of the fundamental sequences; on $t$ and on $\beta$ respectively. This adds a second requirement on the correction function where intuitively, its growth over the tree ordinals should be large enough so that the norm bound in the above lemma will be preserved on $x+Cr(t)$. 
\begin{lemma}[EA]\label{lemm: Norm on term step}
$N(o(t[x])) \leq N(o(t)[x + Cr(t)]) + x + Cr(t) -2$ and $o(t[x]) \leq o(t)[x+Cr(t)]$ for $x\geq 2$.
\end{lemma}
\begin{proof}
We will look only at the cases where $t\in {\sf Lim}$, as the case for successor $t$ is straightforward.
Let $t= s + r + \omega^{t_0}$ with $o(r+\omega^{t_0}) = o( \omega^{t_0})$ and $o(s+\omega^{t_0})= o(s)+o(\omega^{t_0})$, and proceed by induction on the structure of $t$.
\begin{enumerate}
    \item ($t_0 = t'+1$). Then $N(o(t[x])) = N o(s) + N( \omega^{o(t')} (k+x)) \leq N o(s) + N( \omega^{o(t')} (Cr(t)+x)) = N(o(t)[x+Cr(t)])$, since $k \leq N r \leq Cr(t)$ and $o(t_0)[x] = o( t') = o(t_0 [x])$. That $o(t[x]) \leq o(t)[x+Cr(t)]$ is derived mutatis mutandis.
    \item ($t_0 \in {\sf Lim}$). Then $N(o(t[x]))= N(o(s)) + N(o(r+\omega^{t'[x]})= N(o(s)) + N(r' + \omega ^{o(t'[x])})$. By IH, $N(o(t'[x])) \leq N(o(t')[x + Cr(t')]) + x + Cr(t') -2$ and since $N(r') + Cr(t') \leq Nr + Cr(t') \leq Cr(t)$, we get $N(o(t[x])) \leq N(o(t)[x + Cr(t')]) + x + Cr(t) -2 \leq N(o(t)[x + Cr(t)]) + x + Cr(t) -2$. Additionally, by IH $o(t_0 [x]) \leq o(t_0)[x+Cr(t_0)]$. If $r' = \omega^{r_0} + r_1$, then $r_0 < o(t_0)$ and $N r_0 \leq Cr(t) \leq N(o(t_0)[x+Cr(t)] + x + Cr(t) -2$. So by Lemma \ref{lemmNormBound}, $o(r_0) < o(t_0) $, hence $o(t[x]) \leq o(t)[x+Cr(t)]$.
\end{enumerate}
\end{proof}

This amounts to the following theorem.
It gives a tight connection between Hardy hierarchies based on standard ordinals and on tree ordinals, and we expect it to find many applications beyond the current work.

\begin{theorem}[EA]\label{thm: Hardy on ordinals give hardy on tree ordinals}
Let $t \in \mathcal{T}$. If $o(t)\leq \beta$, if $n+N(o(t))-N\beta\leq m-2$ and $n+Cr(t)< m-2$ and $H_{\beta} (m)$ is defined then $H_t$ is defined and $H_t (n) \leq  H_{\beta} (m)$.
\end{theorem}

\begin{proof}

Since $H_\beta (m)=M$ is defined, it has an evaluation sequence $e=\big ( (\beta_i,m_i) \big )_{i\leq k}$. 

\noindent {\bf Claim:}
For all $e$, $M$, if $o(t)\leq \beta$ and $n+N(o(t))-N\beta\leq m-2$ and $H_{\beta} (m) = M$ has evaluation sequence $e$ then $H_t (n) \leq  H_{\beta} (m)$ with an evaluation sequence $v\leq e^{e}$ of length $\leq k+1$.

By bounded induction on $e$.

\begin{enumerate}

\item ($\beta>o(t)$).
By Lemma \ref{lemmNormBound}, $\beta[m]\geq o(t)$.
The evaluation sequence for $H_{\beta[m]}(m)$ ($H_{\beta[m]}(m+1)$ if $\beta$ is a successor ordinal ) is $e':=\big ( (\beta_{i},m_{i}) \big )_{i\leq k-1}.$
By IH on $e'<e$, the claim follows by applying it to $e', m$, $t$ and $\beta[m]$.

\item ($\beta=o(t)$). 
Since $Cr(t[n]) \leq Cr(t)$, by Lemma \ref{lemm: Norm on term step}, we obtain that $o(t[n]) \leq \beta[m] $ and $N(o(t[n])) \leq N(\beta[m]) + m -2$ so by IH, $H_{t[n]}(n) \leq H_{\beta[m]}(m)$ ( $H_{t[n]}(n+1) \leq H_{\beta[m]}(m+1) $ if $\beta $ is a successor) with evaluation sequences $v'\leq e'^{e'}$. Then $H_t(n) \leq H_\beta (m)$ and $v \leq e^e$.

\end{enumerate}

\end{proof}

\subsection{Hardy functions and worms}
In this subsection we will be looking at the details in the differences between the Hardy functions $H_t$ that we defined in subsection \ref{subsect: Hardy functions} and the Hardy functions on the worms $h_A$. First we will transform the worms into the corresponding tree ordinals. The specific definition was chosen based on the exact structural behavior they exhibit.
\begin{definition}
Given a worm $A = n_0 \ldots n_{k-1}$ the tree ordinal corresponding to the worm is given as follows: 
\begin{itemize}
    \item $\tau (\monid) = 0 $;
    \item $\tau (  0^{k_0} A_1 ^+ 0^{k_1} \ldots 0^{k_{n-1}} A_n ^+ 0^{k_{n}}) = k_{n} + \omega^{\tau (A_n)} + k_{n-1} + \ldots + \omega^{\tau (A_0)} + k_0  $, where $A_i \neq \monid$.
\end{itemize}
We can also define the ordinal corresponding to a worm $A$ as:
\[
o(A) := o(\tau(A)).
\]
\end{definition}
Observe that $o(n) = \omega_n$, where $\omega_n$ is defined inductively as
\begin{itemize}
    \item $\omega_0=1$;
    \item $\omega_{n+1} = \omega^{\omega_n}$.
\end{itemize} 
On a more technical note, $\tau (A0B) = \tau(B) + 1 + \tau(A)$ for any worms $A,B$. 

A behavior we find in worms is that under this structural treatment, their fundamental sequences become slightly more involved than the fundamental sequences for tree ordinals:
\begin{itemize}
    \item $0\llbracket x \rrbracket = 0$;
    \item $(t+1) \llbracket x \rrbracket = t$;
    \item $(t+ \omega) \llbracket x \rrbracket = t + x$;
    \item $(t+\omega^{s+1}) \llbracket x \rrbracket = t+ (\omega^s +1) x$, where $s\neq 0$;
    \item $(t+\omega ^\lambda)\llbracket x \rrbracket = t+ \omega^{\lambda\llbracket x \rrbracket}$.
\end{itemize}



This translation of the worms into the corresponding tree ordinals as well as the corresponding fundamental sequences is natural in the following sense.

\begin{lemma}[EA]
$\tau (A\llbracket x\rrbracket) = \tau (A) \llbracket x+1\rrbracket$ for every worm $A$ and natural number $n$.
\end{lemma}
\begin{proof}
By induction on the number of the leftmost element of the worm.
\begin{enumerate}
    \item ($A= 0 B$). Then, $\tau (A) = \tau (B) + 1$ and hence it is clear.
    \item ($A = 10B$). Then,
    \begin{align*}
    \tau(A\llbracket x \rrbracket)& = \tau(0^{x+1} 0 B) 
     = \tau (B) + 1 + \omega\llbracket x+1 \rrbracket = \tau(A) \llbracket x+1 \rrbracket.
    \end{align*}
    \item ($A= 1C^+0B$). In this case,
\begin{align*}    
    \tau (A\llbracket x \rrbracket) & = \tau((0C^+)^{x+1}0B) = \tau(B) + 1 + (\omega^{\tau(C)}+1)^{x+1} \\
&    = \tau(B)+1+ \omega^{\tau(C0)}\llbracket x+1 \rrbracket = \tau(A)\llbracket x +1 \rrbracket.
\end{align*}
    \item ($A = mC^+ 0B$ where $1<m$). Then
    \[\tau (A \llbracket x  \rrbracket) = \tau(B) + 1+ \omega ^{\tau( C(m-1)\llbracket x\rrbracket)} .\]
    By the IH,
    $\tau( C(m-1)\llbracket x\rrbracket) = \tau( C(m-1))\llbracket x+1 \rrbracket,$
    so  
\[    
    \tau(B) + 1+ \omega ^{\tau( C(m-1)\llbracket x\rrbracket)}  =\tau(B) + 1+ \omega ^{\tau( C(m-1))}\llbracket x+1\rrbracket, \]
    and hence
    $\tau (A\llbracket x \rrbracket) = \tau(A) \llbracket x+1 \rrbracket$.
\end{enumerate}
The proof for the cases where $A=B^+$ for some $B$ follows mutatis mutandis.
\end{proof}
The corresponding Hardy function on the tree ordinals is defined accordingly.
\begin{definition}\label{def: hardy on worm based tree ordinals}
Define the following Hardy function on tree ordinals:
\begin{itemize}
    \item $h_0 (x) = x$;
    \item $h_{t+1}(x) = h_t(x+1)$;
    \item $h_t (x) = h_{t\llbracket x \rrbracket}(x+1)$, where $t$ is a limit term.
\end{itemize}
Then $h_{A} (x) + x + 1 = h_{\tau(A)}(x+1)$.
\end{definition}

The difference between the fundamental sequences inspired by the worms $\cdot\llbracket \cdot \rrbracket$ to the more natural ones $\cdot [\cdot ]$ is that the former may contain additional instances of $+1$. Further applications of the corresponding sequences will maintain this overall distinction, where the two terms will be similar, but with instances of $+1$ possibly added when applying $\cdot\llbracket \cdot \rrbracket$ rather than $\cdot[\cdot]$.
We formalize this via two relations: $tRt'$ means that $t$ looks like $t'$ but possibly with more instances of $+1$, while $\tilde R $ is defined analogously but does not allow for the addition of $+1$ terms at the end.
Let us make this precise.

\begin{definition}
    We define the reduction relation $R$ inductively on the structures of tree ordinals as follows: Given $t= k_{-1} + \omega^{t_0} + k_0 + \omega^{t_1} + \ldots + k_{n-1} + \omega^{t_n} + k_{n} $, then
    \[tR\,t' 
\ \ \Leftrightarrow \ \   t' = k_{-1} + \omega^{t'_0} + (k_0-e_0) + \omega^{t'_1} + \ldots + (k_{n-1}-e_{n-1}) + \omega^{t'_n} + (k_{n}-e_{n}),\]
 where $e_i \in 2$ and $t_i R \,t'_i $ for all appropriate $i$'s.\\
    The end-agreeable reduction relation $\tilde{R}$ is defined as follows: Given $t= k_{-1} + \omega^{t_0} + k_0 + \omega^{t_1} + \ldots + k_{n-1} + \omega^{t_n} + k_{n} $, then
    \[t\tilde{R}\,t'
\ \ \Leftrightarrow    \ \
    t' = k_{-1} + \omega^{t'_0} + (k_0-e_0) + \omega^{t'_1} + \ldots + (k_{n-1}-e_{n-1}) + \omega^{t'_n} + (k_{n}),\]
    where $e_i \in 2$ and $t_i R \,t'_i $ for all appropriate $i<n$ and $t_n \tilde{R}\, t'_n$.
\end{definition}
The relation $\tilde{R}$ has the additional property that if $t\tilde{R}\,t'$, then $o(t) = o(t')$. However the $\tilde{R}$ relation is not closed under the corresponding fundamental sequences and we will instead typically land into the $R$ relation.
\begin{lemma}[EA]
If $t\tilde{R}\, t'$, then $t\llbracket x \rrbracket R \,t'[x]$ for any natural number $x$.
\end{lemma}
\begin{proof}
If $t,t'$ are both successors, then it is clear. So assume they are limit tree ordinals and we prove the lemma by induction on the structures of $t$ and $t'$.
\begin{enumerate}
    \item ($t= s + \omega^{r+1}$). Then $t\llbracket x\rrbracket = s+ (\omega^r +1)x$ and $t'[x]= s'+ (\omega^{r'} + (1-1))x$ and since $sR\,s'$ and $rR\,r'$, we also have that $tR\,t'$.
    \item ($t= s+ \omega^r$). Then $t\llbracket x \rrbracket = s + \omega^{r\llbracket x \rrbracket}$, by the IH $r\llbracket x \rrbracket R\, r [x]$ and hence $t\llbracket x \rrbracket R\, t[x]$.
\end{enumerate}
\end{proof}

Overall, passing from $\cdot[\cdot]$ to $\cdot \llbracket\cdot\rrbracket$ will produce a few minute increases in various places. The idea is that for an appropriate constant $c$, the Hardy function $H_t(x+c)$ will leave us room to perform corrections over these changes appearing in $h_t(x)$. We first make an evaluation to turn the typical additional $c$ copies produced into an exponential increase.
\begin{lemma}[EA]
If $H_{\omega^s}(x)$ is defined, $x>0$ and $s\neq 0$, then $2x \leq H_{\omega^s}(x)$.\\
If $H_{\omega^s [z]}(x)$ is defined for some natural number $z>1$ and with $s \geq 2$, then $2^z x \leq H_{\omega^s [z]}(x)$.
\end{lemma}
\begin{proof}
\begin{enumerate}
\item Since $\omega \leq_1 \omega^{s'}$, there is some $(\omega,y)$ with $y \geq x$, in the evaluation sequence of $H_{\omega^s}(x)$ and since $H_\omega (y) = 2y$, the claim follows.

\item By induction on $s$.
\begin{enumerate}
    \item ($s=r+1$). Then we show by bounded induction on $z$ that $2^z x \leq H_{\omega^s[z]}(x)$.
    \item ($s \in \sf Lim$). Then $k+\omega \leq_1 s[z]$ for some $k$. Specifically we would have $k+\omega = s[z][1]^y$ where $y \geq z-1$. Then $ (k+\omega,v)$ is an element of the sequence defining $H_{\omega^s[z]}(x)$ for some $v \geq x+y$. With one more step in the Hardy function, we end up in the previous case.
\end{enumerate}

\end{enumerate}
\end{proof}
This comes as a second step of the previous lemma to get the comparison in the argument of the Hardy function on the tree ordinals $h$. The monotonicity arguments we use in the proof below can be proved similarly to the ones for the Hardy function on worms.
\begin{lemma}[EA]
If $H_{t}(x+c)$ is defined, $t = s + \omega^r $ where $o(r)\geq 2$ and $c\geq 1$, then 
\[
H_{t[x]}(2^c(x+c)) \leq H_{t}(x+c).
\]
\end{lemma}
\begin{proof}
We can assume that $x\geq 2$ and we prove the theorem for $c=1$. There are two cases to consider.
\begin{enumerate}
    \item ($r=r'+1$). Then $H_t(x+c) = H_{s+\omega^{r'}(x+c)}(x+c) = H_{t[x]}(H_{\omega^{r'}}(x+c)) \geq H_{t[x]}(2^c(x+1))$.
    \item ($r\in {\sf Lim}$). Then $t' = s+ \omega^{r[x] +1} \leq_1 t[x+C] $ and hence there is $y\geq x+c$ such that $(t',y)$ is an element of the sequence defining $H_t(x+c)$. Hence $H_t(x+c) = H_{t'} (y) \geq H_{t[x]}(2^c y) \geq H_{t[x]}(2^c (x+1))$.
\end{enumerate}
\end{proof}

The difference in the two fundamental sequences can be quantified by the differences in their norms. Note that this gives us a multiplicative bound while an increase in the argument of a Hardy function gives an exponential bound.
\begin{lemma}[EA]
Given $t,t' \in {\sf Lim}$ with $t\tilde{R}\,t'$ and natural number $x\geq 1$, then 
\[
N(t\llbracket x\rrbracket)-N(t'[x]) \leq \big  (N(t[x]) - N(t'[x]) \big ) +x  \leq x(Nt - Nt' +1).
\]
\end{lemma}
\begin{proof}
By induction on the structure of $t$ we show that $Nt\llbracket x \rrbracket \leq Nt[x] +x$.
\begin{enumerate}
    \item ($t= s+1$ or $t = s + \omega $). Clear.
    \item ($t= s + \omega^{r+1}$ with $r\neq 0$). Then, $Nt\llbracket x \rrbracket = N \big ( s+ (\omega^r+1)x \big ) = Nt[x] + x$.
    \item ($t= s + \omega^r$ with $r\in \sf Lim$). Then, by IH, $Nr\llbracket x\rrbracket \leq Nr[x] + x $ and so $Nt\llbracket x \rrbracket = Ns + 1 + Nr\llbracket x \rrbracket \leq Nt[x] +x$.
\end{enumerate}
Now we show by induction on the depth of $t,t'$ that $N(t[x]) - Nt'[x] \leq x(Nt-Nt')$.
\begin{enumerate}
    \item ($t=s+1$). Clear.
    \item ($t=s+\omega^{r+1}$). Then $Nt[x] - Nt'[x]= N(s+ \omega^r x) - N(s+ \omega^{r'} x) = Nt - Nt' + N(\omega^r (x-1) - N (\omega ^{r'} (x-1) \leq x(Nt - Nt') $.
    \item ($t= s + \omega^r $ with $r\in \sf Lim$). Then $Nt[x] - Nt'[x] = Ns - Ns' + Nr[x] - Nr'[x] \leq Ns - Ns' + x(Nr - Nr') \leq x(Nt - Nt')$.
\end{enumerate}
\end{proof}

As a result, an increase on the argument by $c=2$ is sufficient. $1$ to counterbalance the increase of the argument of $h$ in the limit stages and $1$ for the difference in the fundamental sequences.
\begin{theorem}[EA]\label{thm: hardy on tree ordinals bounds worm based hardy}
Assume that $m\geq n+2$ and $H_t(m)$ is defined, then $h_t(n)$ is defined and $h_t(n) \leq H_t(m)$.
\end{theorem}
\begin{proof}
Since $H_t (m)=M$ is defined, it has an evaluation sequence $e=\big ( (t_i,m_i) \big )_{i\leq k}$.

\noindent{\bf Claim:} If $t\tilde{R}\,t'$, $m \geq N(t)- N(t') + n + 2$ and $H_{t'} (m) = M$ has evaluation sequence $e$, then $h_t (n)$ is defined with $h_t (n) \leq H_{t'}(m)$ and evaluation sequence $v\leq e^e$ of length $\leq k+1$.\\
By bounded induction on $e$.
\begin{enumerate}
    \item ($t=s+1$). Clear.
    \item ($t = s+ \omega$). Then $Nt\llbracket n \rrbracket - Nt'[n] = Nt - Nt' $ and hence we immediately have by IH that $ h_t(n) =h_{t\llbracket n \rrbracket} (n+1) \leq H_{t'[n]}(2m-n) = H_{t'}(m)$.
    \item ($t=s+\omega^{r+1}$ and $r\neq 0$). Then $t\llbracket n \rrbracket = s + (\omega^r +1)n$. Then
\begin{align*} 
 N(t &\llbracket n\rrbracket)-N(t'[n]+1) +n  =N(t\llbracket n\rrbracket)-N(t'[n]) +n-1 \\
 & \leq  (N(t[n]) - N(t'[n])) +2n -1 \leq n (Nt - Nt' +2) -1 \leq m 2^{Nt-Nt' +2} -1\leq m',
\end{align*}
     where $m'$ is such that $(m', t[n]+1)$ is an element of the sequence defined by $H_t(m)$ and by the IH, $h_{t\llbracket n \rrbracket}(n+1) \leq H_{t[n]+1}(m')$.
    \item ($t=s+\omega^r$ with $r\in \sf Lim$). Then there is $t_0 \leq_1 t'[m]$ such that $t_0[1] = t'[n]$. As before,
\begin{align*}    
    N(t & \llbracket n\rrbracket)-N(t_0) +n= N(t\llbracket n\rrbracket)-N(t'[n]) +n-1\\
    & \leq  (N(t[n]) - N(t'[n])) +2n -1 \leq n (Nt - Nt' +2) -1 \leq m 2^{Nt-Nt' +2} -1\leq m',
\end{align*}
where $m'$ is such that $(m', t_0)$ is an element of the sequence defined by $H_t(m)$, so by the IH, $h_{t\llbracket n \rrbracket} (n+1) \leq H_{t_0}(m)$.
\end{enumerate}
\end{proof}

\subsection{Reflection and worm battles}
The remaining step is to move from reflection principles into the corresponding assertion that a class of Hardy functions is total. We define the following iterated reflection principles on a given ordinal $\alpha$.
\begin{definition}
	Given a c.e.~base theory $T$, we define 
	\[
	\Pi_n\textnormal{-}R^{\alpha}(T):= T + \{ \Box_{\bigcup_{\beta<\alpha}\Pi_n\textnormal{-} R^{\beta}(T)} \pi \to \pi \mid \pi \in \Pi_n\}.
	\]
\end{definition}
The $\Pi_n\textnormal{-}R^{\alpha}(T)$ formulae are defined with use of the fixed point theorem and are monotonous over the ordinals $\alpha$:
\begin{align*}
    \text{if } \alpha < \beta \text{ then } \Pi_n\textnormal{-}R^{\alpha}(T) \vdash \Pi_n\textnormal{-}R^{\beta}(T).
\end{align*}

We can then use the following remark to transition from reflection principles over $\ea$ to those over $\ea^+$ since the strength of the corresponding reflection of the theories $\is_n$ is sufficient.
\begin{rem} \label{rem: RFN equivalence between U and U+ psi}
	If $ \ea \subseteq T $ and $ \psi \in \Sigma_n $ is without free variables, where $ T \vdash \psi $ and $ U $ is a subtheory of $ T $, then
	$$ T \vdash \Pi_n\text{-}\rfn(U + \psi) \leftrightarrow \Pi_n\text{-}\rfn(U)
	$$
\end{rem}
\begin{proof}
	Let $ \varphi \in \Pi_n $ which without loss of generality has at most one free variable, then since $ \psi \rightarrow \varphi  $ is a $ \Pi_n $ formula,
	\begin{align*}
		T \vdash \forall\, x \  \big( \Box_{U + \psi} \varphi (\dot{x}) \rightarrow \varphi(x) \big) &\leftrightarrow \Big( \psi \rightarrow \forall\, x \  \big( \Box_U ( \psi \rightarrow \varphi (\dot{x})) \rightarrow \varphi(x) \big) \Big)\\
		&\leftrightarrow  \forall\, x \  \big( \Box_U ( \psi \rightarrow \varphi (\dot{x})) \rightarrow (\psi \rightarrow \varphi(x)) \big); 
	\end{align*}
	therefore, $ T \vdash \Pi_n\text{-}\rfn(U) \rightarrow \Pi_n\text{-}\rfn(U + \psi)  $. The other direction comes from the fact that if $ U_1 \subseteq U_2 $, then $ T \vdash \Box_{U_1} \varphi(\dot{x}) \rightarrow \Box_{U_2} \varphi(\dot{x}) $.
\end{proof}

The following Lemma, in a sense, expresses a decompression of worms into $\alpha$-iterated reflection principles \cite{Beklemishev:2004:ProvabilityAlgebrasAndOrdinals}.

\begin{lemma} \label{lem: relationship between worms and alpha iterated reflection}
	Let $ T $ be a c.e. extension of $ \ea^+ $ whose axioms have logical complexity of $ \Pi_{n+1} $. Then for every worm $ A\in \w_n $, provably in $ \ea^+ $
	\[T + A_T \equiv_n \prfn{n+1}{o(n\downarrow A)}(T).
	\]
\end{lemma}

Now to talk about $ \is_n $ for $ n{\geq}1 $, we first have by Remark \ref{rem: RFN equivalence between U and U+ psi} the following:
\[ 
\ea + \big(\la 1 \ra \la n+1 \ra \top\big)_{\ea}  \equiv \ea^+ + \big(\la 1 \ra \la n+1 \ra\top\big)_{\ea} \equiv \ea^+ + \big(\la 1 \ra \la n+1 \ra\top\big)_{\ea^+}.	
\]
Using this fact along with Lemma \ref{lem: relationship between worms and alpha iterated reflection},
\begin{corollary} \label{cor: from ISn to alpha itterated reflection}
For every natural number $n\geq1$,
\[ 
\ea + 1\textnormal{-}\con(\is_n) \equiv_1 \prfn{2}{o(0n)}(\ea ^+).
\]
\end{corollary}
The $\Pi_2$-conservativity is of interest to us since any assertion of the form $f{\downarrow}$ of interest to us is expressed in a $\Pi_2$-sentence. We will make a roundabout way to get the Hardy functions by first jumping to the so-called \textit{fast growing functions,} which are a bit more closely connected to the corresponding iterated $\Pi_2$-reflection principles.
\begin{definition}
    The fast growing hierarchy is defined as follows:
    \begin{itemize}
        \item $F_0(x) = 2^x _x$;
        \item $F_{\alpha +1}(x) = F_\alpha ^{(x)}(x)$;
        \item $F_\lambda(x) = F_{\lambda[x]}(x).$
    \end{itemize}
\end{definition}

This fast-growing hierarchy can be compared to the Hardy hierarchy via the following.
The proof by transfinite induction is standard and can be formalized in $\ea$ along the lines of the proof of Theorem \ref{thm: Hardy on ordinals give hardy on tree ordinals}.

\begin{lemma}[EA]\label{lem: Fast growing bounding hardy functions}
If $F_\alpha{\downarrow} $, then $H_{\omega^{3+\alpha}} {\downarrow}$ and $F_\alpha (x) \leq H_{\omega^{3+\alpha}}(x+3) \leq F_\alpha (x+4)$ for every ordinal $\alpha$. 
\end{lemma}
Additionally, the fast growing functions correspond to the iterated reflection principles in the following manner \cite{Beklemishev:2003:ProofTheoreticAnalysisByIteratedReflection}:

\begin{theorem}[$\ea ^+$]\label{thm: itterated reflection to totality of F}
For every ordinal $\alpha < \varepsilon_0$ we have $ \prfn{2}{\alpha}(\ea ^+) \equiv \ea + \{ F_\beta {\downarrow} : \beta < \alpha\} $.
\end{theorem}

Therefore from $1\text{-}\con(\is_n)$, we can assert that the corresponding $H_{\omega^\alpha}$ functions are total and hence, through our theorems comparing the verious hardy functions, we can prove $\ewd^{n+1}$:
\begin{theorem}\label{theorem:fineStructureForWorms}
$\ea + 1\text{-}\con(\is_n) \vdash \ewd^{n+1}.$
\end{theorem}

\begin{proof}
By Corollary \ref{cor: from ISn to alpha itterated reflection} and Theorem \ref{thm: itterated reflection to totality of F},
\[
\ea + 1\textnormal{-} {\sf Con}(\is_n) \equiv_1 \ea+1\textnormal{-} {\sf Con}(\forall \alpha< \omega_n \ F_\alpha {\downarrow}).
\]
As the assertion of totality of the fast growing functions and the Hardy functions respectively are $\Pi_2$-sentences, and since $1\textnormal{-} {\sf Con}$ is equivalent to $\Pi_2$-reflection, we have that 
\[
\ea + 1\textnormal{-} {\sf Con}(\forall \alpha< \omega_n \ F_\alpha {\downarrow}) \vdash \forall \alpha< \omega_n \ F_\alpha {\downarrow}
\]
and consequently, using Lemma \ref{lem: Fast growing bounding hardy functions} we obtain
\[
\ea + 1\textnormal{-}\con(\is_n) \vdash \forall \alpha<\omega_n \ H_{\omega^\alpha} {\downarrow}.
\]
Then by Theorems \ref{thm: Hardy on ordinals give hardy on tree ordinals}, \ref{thm: hardy on tree ordinals bounds worm based hardy} and Definition \ref{def: hardy on worm based tree ordinals} give the following implications over $\ea$:
\[
\forall \alpha<\omega_n \ H_{\omega^\alpha} {\downarrow} \vdash \forall \alpha<\omega_{n+1} \ H_{\alpha} {\downarrow} \vdash \forall\alpha<\omega_{n+1} \ h_{\alpha} {\downarrow} \vdash \forall A\in W^{n+1} \ h_{A} {\downarrow}.
\]
The last one implying $\ewd^{n+1}$ by definition over $\ea$.
\end{proof}

\section{Concluding remarks}

We have shown that $\glp_\Lambda$ is sound for the transfinite notions of provability studied by Beklemishev and Pakhomov~\cite{BeklemishevPakhomov:2019:GLPforTheoriesOfTruth}, and with this we have shown that a natural estension of the {\em Every Worm Dies} principle is independent of $\aca$.
Likewise, we have shown that restricted versions of this principle are equivalent to the theories $\is_n$.
The proof of the latter required a detour through Hardy functions and fast-growing hierarchies, in particular yielding a non-trivial comparison between Hardy functions based on ordinals and those based on tree ordinals which should be of independent interest.

Stronger theories of second order arithmetic should also be proof-theoretically equivalent to reflection up to a suitable ordinal $\Lambda$.
These equivalences may then be used to provide new variants of $\ewd$ independent of stronger theories of second order arithmetic, including theories related to transfinite induction or iterated comprehension.
We expect that this work will be an important step in this direction.

%
%
%
\bibliographystyle{splncs04}
\bibliography{References}

\begin{thebibliography}{10}
\providecommand{\url}[1]{\texttt{#1}}
\providecommand{\urlprefix}{URL }
\providecommand{\doi}[1]{https://doi.org/#1}

\bibitem{Beklemishev:1997:InductionRules}
Beklemishev, L.D.: Induction rules, reflection principles, and provably
  recursive functions. Annals of Pure and Applied Logic  \textbf{85},  193--242
  (1997)

\bibitem{Beklemishev:2003:ProofTheoreticAnalysisByIteratedReflection}
Beklemishev, L.D.: Proof-theoretic analysis by iterated reflection. Archive for
  Mathematical Logic  \textbf{42},  515--552 (2003)

\bibitem{Beklemishev:2004:ProvabilityAlgebrasAndOrdinals}
Beklemishev, L.D.: Provability algebras and proof-theoretic ordinals, {I}.
  Annals of Pure and Applied Logic  \textbf{128},  103--124 (2004)

\bibitem{Beklemishev:2005:Survey}
Beklemishev, L.D.: Reflection principles and provability algebras in formal
  arithmetic. {Uspekhi Matematicheskikh Nauk}  \textbf{60}(2),  3--78 (2005),
  in Russian. English translation in: \emph{Russian Mathematical Surveys},
  60(2): 197--268, 2005.

\bibitem{BeklemishevFernandezJoosten:2014:LinearlyOrderedGLP}
Beklemishev, L.D., Fern\'andez-Duque, D., Joosten, J.J.: {On provability logics
  with linearly ordered modalities}. Studia Logica  \textbf{102},  541--566
  (2014)

\bibitem{BeklemishevPakhomov:2019:GLPforTheoriesOfTruth}
Beklemishev, L.D., Pakhomov, F.N.: Reflection algebras and conservation results
  for theories of iterated truth. arXiv:1908.10302 [math.LO]  (2019)

\bibitem{CordonFernandezJoostenLara:2017:PredicativityThroughTransfiniteReflection}
Cord\'on~Franco, A., Fern\'andez-Duque, D., Joosten, J.J., Lara~Mart\'{\i}n,
  F.: Predicativity through transfinite reflection. Journal of Symbolic Logic
  \textbf{82}(3),  787--808 (2017)

\bibitem{Enayat2019-ENATDA}
Enayat, A., Pakhomov, F.: Truth, disjunction, and induction. Archive for
  Mathematical Logic  \textbf{58}(5-6),  753--766 (2019)

\bibitem{Fernandez:2012:TopologicalCompleteness}
Fern{\'a}ndez-Duque, D.: The polytopologies of transfinite provability logic.
  Archive for Mathematical Logic  \textbf{53}(3-4),  385--431 (2014)

\bibitem{FernandezJoosten:2012:Hyperations}
Fern\'andez-Duque, D., Joosten, J.J.: Hyperations, {V}eblen progressions and
  transfinite iteration of ordinal functions. Annals of Pure and Applied Logic
  \textbf{164}(7-8),  785--801 (2013)

\bibitem{FernandezJoosten:2013:ModelsOfGLP}
Fern\'andez-Duque, D., Joosten, J.J.: Models of transfinite provability logics.
  Journal of Symbolic Logic  \textbf{78}(2),  543--561 (2013)

\bibitem{FernandezJoosten:2014:WellOrders}
Fern\'andez-Duque, D., Joosten, J.J.: Well-orders in the transfinite
  {J}aparidze algebra. Logic Journal of the {IGPL}  \textbf{22}(6),  933--963
  (2014)

\bibitem{FernandezJoosten:2018:OmegaRuleInterpretationGLP}
Fern\'andez-Duque, D., Joosten, J.J.: The omega-rule interpretation of
  transfinite provability logic. Annals of Pure and Applied Logic
  \textbf{169}(4),  333--371 (2018)

\bibitem{Halbach:2014:AxiomaticTheoriesTruth}
Halbach, V.: Axiomatic Theories of Truth. University of Oxford (2014)

\bibitem{Japaridze:1988}
Japaridze, G.: The polymodal provability logic. In: Intensional logics and
  logical structure of theories: material from the {F}ourth {S}oviet-{F}innish
  Symposium on Logic. Metsniereba, Telaviv (1988), in {R}ussian

\bibitem{Joosten:2013:AnalysisBeyondFO}
Joosten, J.J.: {$\Pi^0_1$}-ordinal analysis beyond first-order arithmetic.
  Mathematical Communications  \textbf{18},  109--121 (2013)

\bibitem{Joosten:2016:TuringTaylorExpansion}
Joosten, J.J.: Turing-{T}aylor expansions of arithmetic theories. Studia Logica
   \textbf{104},  1225--1243 (2016)

\bibitem{Joosten:2020:MunchhausenProvability}
Joosten, J.J.: M{\"u}nchhausen provability. Journal of Symbolic Logic
  \textbf{86}(3),  1006--1034 (2021), \url{doi:10.1017/jsl.2021.44}

\bibitem{Leivant:1983:OptimalityOfInduction}
Leivant, D.: The optimality of induction as an axiomatization of arithmetic.
  Journal of Symbolic Logic  \textbf{48},  182--184 (1983)

\bibitem{Papafillipou:2020:MastersThesis}
Papafillipou, K.: Independent combinatoric worm principles for first order
  arithmetic and beyond. Master's thesis, Master of Pure and Applied Logic,
  University of Barcelona (2020),
  \url{http://diposit.ub.edu/dspace/handle/2445/170755}

\bibitem{Pohlers:2009:PTBook}
Pohlers, W.: Proof Theory, The First Step into Impredicativity.
  Springer-Verlag, Berlin Heidelberg (2009)

\bibitem{RathjenArt}
Rathjen, M.: The art of ordinal analysis. In: Proceedings of the International
  Congress of Mathematicians. vol.~2, pp. 45--69. European Mathematical Society
  (2006)

\bibitem{Turing:1939:TuringProgressions}
Turing, A.: Systems of logics based on ordinals. Proceedings of the London
  Mathematical Society  \textbf{45},  161--228 (1939)

\end{thebibliography}
%




\end{document}